\documentclass[12pt]{amsart}
\usepackage{amsfonts}
\usepackage{amsfonts,latexsym,rawfonts,amsmath,amssymb,amsthm}
\usepackage[plainpages=false]{hyperref}

\usepackage{graphicx}

\RequirePackage{color}

\numberwithin{equation}{section}

\newcommand{\beq}{\begin{equation}}
\newcommand{\eeq}{\end{equation}}
\newcommand{\beqs}{\begin{eqnarray*}}
\newcommand{\eeqs}{\end{eqnarray*}}
\newcommand{\beqn}{\begin{eqnarray}}
\newcommand{\eeqn}{\end{eqnarray}}
\newcommand{\beqa}{\begin{array}}
\newcommand{\eeqa}{\end{array}}

\def\p{\partial }

\def\R{\Bbb R}

\def\Om{\Omega}
\def\pom{\p  \Omega}

\def\noo{\noindent}

\newtheorem{prop}{Proposition}[section]
\newtheorem{theo}[prop]{Theorem}
\newtheorem{lem}[prop]{Lemma}

\newtheorem{rem}[prop]{Remark}

\title  {The first boundary value problem for \\
Abreu's equation}

\begin{document}

\author {Bin Zhou}


\address{Bin Zhou, Centre for Mathematics and Its Applications,
Australian National University, Canberra, ACT 0200, Australia; and
\newline
 School of Mathematical Sciences, Peking
University, Beijing, 100871, China.}

\thanks {The author is supported by the State Scholarship
Fund of China.}


\email{bzhou@pku.edu.cn}


\bibliographystyle{plain}

\begin{abstract}
In this paper we prove the existence and regularity of solutions to 
the first boundary value problem for Abreu's equation, 
which is a fourth order nonlinear partial differential equation closely 
related to the Monge-Amp\`ere equation. 
The first boundary value problem can be formulated 
as a variational problem for the energy functional.
The existence and uniqueness of maximizers can be obtained 
by the concavity of the functional.
The main ingredients of the paper are the a priori estimates
and an approximation result, which enable us to prove that
the maximizer is smooth  in dimension 2.

\end{abstract}

\maketitle

\baselineskip=15.8pt
\parskip=3pt

\section {Introduction}

\vskip 10pt

Abreu's equation was first introduced by M.\ Abreu \cite{Ab} in the study of
existence of extremal metrics on toric K\"ahler manifolds.
It is a fourth order equation given by
\beq
\sum_{i, j=1}^n \frac{\partial^2u^{ij}}{\partial x_i\partial x_j}=f
\eeq
where $u$ is a convex function in a bounded domain $\Omega$
in $\mathbb R^n$, $f\in{L^\infty(\Omega)}$,
and $(u^{ij})$ is the inverse matrix of the Hessian $(u_{ij})$.
This equation was later studied by S.\ Donaldon.
In a series of papers \cite{D1, D2, D3, D4},
Donaldon established various a priori estimates for Abreu's equation
and proved the existence of constant scalar curvature metrics
on toric K\"ahler surfaces under the assumption of K-stability.

Abreu's equation can also be written as
\beq
U^{ij}w_{ij}=f,
\eeq
where $(U^{ij})$ is the cofactor matrix of $(u_{ij})$ and
\beq
w=[\det D^2u]^{-1}.
\eeq
The energy functional of Abreu's equation is given by
\beq
J_0(u)=A_0(u)-\int_{\Omega}fu\, dx,
\eeq
where
\beq
A_0(u)=\int_\Omega \log\det D^2u\, dx.
\eeq
We formulate a variational problem for Abreu's equation. Let
\beq
S[\varphi, \Omega]
=\{u\in{C^2(\Omega)\cap C^0(\overline \Omega)} \ |
\  u\  \text{is convex}\  u|_{\p \Omega}=\varphi(x),
Du(\Omega)\subset D\varphi(\overline\Omega)\},
\eeq
where $\varphi$ a smooth, uniformly convex function
defined in a neighborhood of $\overline \Omega$.
The problem is to find a function
$u$ in $S[\varphi, \Omega]$ such that
\beq
J_0(u)=\sup\{J_0(v)\  |\  v\in{S[\varphi, \Omega]}\}.
\eeq

The main result in this paper is as follows.

\begin{theo}
Suppose the domain $\Omega$ is bounded and smooth.
Assume $f\in C^{\infty}(\Omega)\cap L^\infty(\Omega)$.
If $n=2$, there exists a unique, smooth, locally uniformly convex maximizer $u$
of the variational problem (1.7).
\end{theo}

The variational problem (1.7) corresponds to
the {\it first boundary value problem} for equation (1.1),
\beqn
u&=&\varphi \  \  \  \ \text{on} \  \partial \Omega,\\
Du&=&D\varphi \  \  \text{on} \  \partial \Omega.
\eeqn
Indeed, if  we have a classical, locally uniformly convex solution
$u\in {C^4(\Omega)\cap C^1(\overline\Omega)}$ to (1.1), (1.8) and (1.9), 
$u$ will also solve (1.7) uniquely. The uniqueness follows from the
concavity of the functional $A_0$.

\vskip 10pt

A motivation for our investigation of the above problem is that
the study of boundary value problems for elliptic equations has
been a focus of attention since 1950s. 
The Dirichlet problem for Monge-Amp\`ere type equations,
which is somehow related to our boundary condition (1.8) above, 
has been studied by many people, 
see \cite {CNS, GS1, Li, S, TW4, U1}.
The second boundary problem for the Monge-Amp\`ere equation,
which is related to our boundary condition (1.9) above,
has also been studied in \cite {Caf2, Del, U2}.

Another motivation to study the above problem is 
due to the increasing interest in nonlinear fourth order partial differential equations.
In recent years, nonlinear fourth order equations, such as 
the affine mean curvature equation and Willmore surface equation,
have attracted considerable attention.
Abreu's equation is similar to the  affine mean curvature equation,
which is given by
\beq
U^{ij}w_{ij}=f,
\eeq
where
\beq
 w=[\det D^2u]^{-(1-\theta)},  \  \  \theta=\frac{1}{n+2}.
\eeq
When $f=0$, (1.10) is called the affine maximal surfaces equation.
The energy functional of affine mean curvature equation is
\beq
J_\theta(u)=A_\theta(u)-\int_{\Omega}fu\, dx,
\eeq
where
\beq
A_\theta(u)=\int_\Omega [\det D^2u]^\theta \, dx
\eeq
is called {\it affine area functional} \cite{Cal, LR}.
In \cite{TW2, TW5}, N. Trudinger and X.-J. Wang studied
the first boundary value problem for the affine maximail surface equation, 
and the more general affine Plateau problem,
which can also be reduced to a similar variational problem.
In \cite{TW2}, Trudinger and Wang proved the existence
and uniqueness of smooth maximizers of $J_\theta$
in $S[\varphi, \Omega]$ in dimension 2.
Theorem 1.1 above is an analogue to their result.
Very recently, they also obtained the regularity of 
maximizers to the affine Plateau problem in high dimensions \cite{TW5}.

\vskip 10pt

Our proof of Theorem 1.1 is inspired by Trudinger and Wang's
variational approach and their regularity argument
in solving the affine Plateau problem.
But due to the singularity of the function $\log d$ near $d=0$,
the approximation argument in \cite{TW2, TW5}
does not apply directly to our problem.
To avoid this difficulty we introduce in Section 2 a sequence of
modified functionals $J_k$ to approximate $J_0$,
such that the integrand in $J_k$ is H\"older continuous at $d=0$.
We prove the existence and uniqueness of a maximizer
of the functional $J_k$ (Theorem 2.6) in the set
$\overline S[\varphi, \Omega]$,
the closure of $S[\varphi, \Omega]$
under uniform convergence.

The regularity of the maximizer is our main concern.
In Section 3 we establish a uniform (in $k$) a priori
estimates for the corresponding Euler equation of the functional $J_k$.
Unlike the affine maximal surface equation,
Abreu's equation is not invariant under linear transformation
of coordinates $\R^{n+1}$. When we rotate the coordinates in $\R^{n+1}$
we get a more complicated 4th order pde (\S4). 
In Section 4, we establish the uniform (in $k$) a priori
estimates for the equations obtained
after rotation of coordinates in $\R^{n+1}$.

As the maximizer may not be smooth, to apply the a priori estimates
we need to prove that the maximizer can be approximated by smooth solutions.
We cannot prove the approximation for the functional $J_0$ directly
as $\log d$ is singular near $d=0$. But for maximizers of $J_k$,
the approximation can be proved similarly as
for the affine Plateau problem \cite{TW2, TW5}.
The approximation solutions are constructed by considering
the {\it second boundary value problem}, 
namely the Euler equation of $J_k$ (see (2.6)) subject to
\beqn
u&=&\varphi \  \  \text{on} \  \partial \Omega,\\
w&=&\psi \  \  \text{on} \  \partial \Omega.
\eeqn
We can prove the existence of locally smooth solutions 
to the boundary value problem (2.6), (1.14) and (1.15),
in a way similar to that in \cite{TW2, TW5}.
For reader's convenience we include a proof in the Appendix.

The a priori estimates in Sections 3 and 4 rely on the strict convexity of solutions.
In Sections 6 and 7 we are devoted to the proof of the strict convexity of solutions.
The proof for one case is similar to that for
affine mean curvature equation in \cite{TW1, TW2}
and is included in Section 6.
But the proof for the other case uses the a priori estimates,
the Legendre transform and in particular a strong approximation
(Theorem 7.1) and is contained in Section 7.
\footnote{This paper was submitted to a journal for publication in June.
Recently, Chen-Li-Sheng posted a related paper \cite{CLS}. 
In their paper, the boundary value problem for (1.1)
with $u=\varphi$, $Du=\infty$ and $w=\infty$ was studied. 
They use solutions to the second boundary value  
problem of Abreu's equation directly as the approximating solutions.
Their approach does not apply
to the case considered in this paper.}

 \vskip 10pt

\noo{\bf Acknowledgement}
The author would like to thank 
Xu-Jia Wang for many inspiring discussions on this problem. 
He would also like to thank Xiaohua Zhu
and Neil Trudinger for their support and interest in the problem.

\vskip 20pt

\section{A modified functional}

\vskip 10pt

In this section we introduce a modified functional $J$ and
prove the existence and uniqueness of a maximizer of $J$.

We begin with some terminologies.
Let $u$ be a convex function in a domain
$\Om\subset\R^n$ and $z\in\Om$ be an interior point.
The {\it normal mapping} of $u$ at $z$, $N_u(x)$, is the set of gradients of the
supporting functions of $u$ at $x$, that is
$$N_u(x)=\{p\in\mathbb R^n\  |  \
u(y)\geq u(x)+p\cdot (y-x) \}.$$
For any subset $\Om'\subset\Om$, 
denote $N_u(\Om')=\bigcup_{x\in\Om'} N_u(x)$. 
If $u$ is $C^1$,
the normal mapping $N_u$ is exactly the gradient mapping $Du$.

For a convex function $u$ on $\Omega$,
the {\it Monge-Amp\`ere measure} 
$\mu[u]$  is a Radon measure given by
$$\mu[u](E)=|N_u(E)|$$
for any Borel set $E$. By a fundamental result of Aleksandrov,
$\mu[u]$ is weekly continuous with respect to the convergence of
convex functions \cite{P, TW3}.
It follows that if $\{u_j\}$ converges to $u$ in $L^1_{loc}$,
then for any closed $E\subset \Omega$,
\beq
\limsup_{j\rightarrow \infty}\mu [u_j](E)\leq \mu[u](E).
\eeq

\vskip 10pt

Since the set $S[\varphi, \Omega]$ is not closed,
we introduce
\beq
\overline S[\varphi, \Omega]
=\{u\in{C^0(\overline\Omega)} \ |
\  u\  \text{is convex}\  u|_{\p \Omega}=\varphi(x),
N_u(\Omega)\subset D\varphi(\overline\Omega)\}.
\eeq
Note that $\overline S[\varphi, \Omega]$ is closed under the
locally uniform convergence of convex functions.
In \cite{ZZ}, we proved that $A_0$ is well defined and
upper semi-continuous in another set of convex functions.
By a similar argument, we can also prove that
$A_0$ is well defined and
upper semi-continuous in $\overline S[\varphi, \Omega]$,
which implies the existence of a maximizer of $J_0$
in $\overline S[\varphi, \Omega]$.

To apply the a priori estimates to the maximizer,
we need a sequence of smooth solutions to Abreu's equation to
approximate the maximizer.
Since the penalty method in \cite{TW2} does not apply to $J_0$,
we must have a sequence of modified smooth approximation solutions.
For this purpose, we consider a functional of the form
\beq
J(u)=A(u)-\int_{\Omega}fu\, dx,
\eeq
where
\beq
A(u)=\int_\Omega G(\det D^2u)\,dx.
\eeq
Here $G(d)=G_\delta(d)$ is a smooth concave function on $[0, \infty)$
which depends on a constant $\delta\in (0, 1)$
and satisfies the following conditions.

(a) $G(d)=\log d$ when $d\geq \delta$.

(b) $G'(d)>0$ and there exist constants $C_1, C_2>0$
independent of $\delta$ such that for any $d>0$
\beqs
&&G''(d)\geq -C_1d^{-2}, \\ [5pt]
&&\left|\frac{dG'''(d)}{G''(d)}\right|\leq C_2.
\eeqs

(c) The function $F(t)=G(d)$, where $t=d^{\frac{1}{n}}$,
is  smooth in $(0, +\infty)$ and  satisfies
 $$\begin{array}{l}
 F(0)> -\infty, \  \    F''(t)<0,\\[8pt]
  {\lim}_{t\rightarrow 0} F'(t)=\infty, \  \  \lim_{t\rightarrow 0} tF'(t)\leq C_3,
 \end{array}$$
 where $C_3$ is a positive constant.

\begin{rem} \ \newline
(i)  The condition $F''(t)<0$ in (c)
implies that the functional $A(u)$ is concave.
\newline
(ii) The concavity of $F$, $F''(t)<0$,
 is equivalent to $dG''(d)+\frac{n-1}{n}G'(d)<0$;
and $ \lim_{t\rightarrow 0}F'(t)=\infty$ is equivalent to
$d^{\frac{n-1}{n}}G'(d)\rightarrow \infty$
 as $d\rightarrow 0$.
 \newline
(iii) We point out the existence of functions $G$
satisfying properties (a)-(c) above.
A function in our mind is
\beq
G(d)=
\begin{cases}
\frac{\delta^{-\theta}}{\theta(1-\theta)}d^{\theta}
-\frac{\theta\delta^{-1}}{1-\theta}d
+\log\delta-\frac{1+\theta}{\theta}, &\  d<\delta , \\
\log d, &\ d\geq\delta,
\end{cases}
\eeq
where $\theta=\frac{1}{n+2}$.
One can check that $G\in{C^{2,1}(0, \infty)}$ and
$C^3$ except at $d=\delta$.
It is easy to see that $G$ satisfies (a) and (c).
We can also check that $G$ satisfies (b) except at $d=\delta$.
Hence, we can always mollify $G$ to have
a sequence of smooth functions satisfying
the properties (a)-(c) to approximate it.
\end{rem}


\vskip 10pt

The Euler equation of the functional $J$ is
\beq
U^{ij}w_{ij}=f,
\eeq
where
\beq
w=G'(\det D^2u)
\eeq
and $(U^{ij})$ is the cofactor matrix of $D^2u$.

\begin{rem}
Equation (2.6) is invariant under unimodular linear transformation.
If we make a general non-degenerate linear transformation
$T: \ y=Tx$ and let  $\tilde u(y)=u(x)$,
then $\tilde u(y)$ is a solution of
$$\tilde U^{ij}\tilde w_{ij}=f, \
\tilde w=\tilde G'(\det D^2\tilde u),$$
where $\tilde G(\tilde d)=G(|T|^2 \tilde d)$, $\tilde d=\det D^2\tilde u$.
Here $\tilde G$ is a smooth concave function satisfying (a), (b), (c) with
$\tilde\delta=|T|^{-2}\delta$, $\tilde C_1=C_1$,
$\tilde C_2=C_2$, $\tilde C_3=C_3$.
\end{rem}

Now we study the existence and
uniqueness of maximizers to the functional $J(u)$.
The treatment here is same as that in \cite{TW2, ZZ},
so we will only sketch the proof.

First, we extend the functional $J$ to
$\overline S[\varphi, \Omega]$. It is clear that the linear part
in $J$ is naturally well-defined. It suffices to
extend $A(u)$ to $\overline S[\varphi, \Omega]$.
Since $u$ is convex, $u$ is almost everywhere twice-differentiable,
i.e., the Hessian matrix $(D^2u)$ exists almost everywhere.
Denote the Hessian matrix by $(\partial^2u)$ at those
twice-differentiable points in $\Omega$.
As a Radon measure, $\mu[u]$ can be decomposed into
a regular part and a singular part as follows,
$$\mu[u]=\mu_r[u]+\mu_s[u].$$
It was proved in \cite{TW2} that the regular part $\mu_r[u]$
can be given explicitly by
$$\mu_r[u]=\det \partial^2u\, dx$$
and $\det \partial^2u$ is a locally integrable function.
Therefore for any $u\in{\overline S[\varphi,\Omega]}$, we can define
\beq
A(u)=\int_\Omega G(\det \partial^2u)\, dx.
\eeq

Next, we state an important property of $A(u)$.
For any Lebesgue measurable set $E$,
by the concavity of $G$ and Jensen's inequality,
\beqn
\int_E G(\det \partial^2u)\,dx
&\leq& |E| G\left(\frac{\int_E\det \partial^2u\, dx}{|E|}\right)\\
&\leq& |E|G(|E|^{-1}\mu[u](E)).\nonumber
\eeqn
By the assumption (a), $d^{-1}G(d)\rightarrow 0$ as $d\rightarrow \infty$.
Note that $G$ is bounded from below.
So the above integral goes to $0$ as $|E|\rightarrow 0$.
With this property, we have an approximation result for the functional $A(u)$.
For $u\in{\overline S[\varphi, \Omega]}$,
let
$$u_h(x)=h^{-n}\int_{B_1(0)}\rho(\frac{x-y}{h})u(y)\,dy,$$
where $h>0$ is a small constant and
$\rho\in C^\infty_0 (B_1(0))$
with $\int_{B_1(0)}\rho=1$.
Suppose that $u$ is defined in a neighborhood of $\Omega$ such that
$u_h$ is well-defined for any $x\in{\Omega}$. A fundamental result is
that $(D^2u_h)\rightarrow (\partial^2u)$
almost everywhere in $\Omega$ \cite{Z}.
Combining it with (2.9), 
we have therefore obtained as in \cite{TW1},

\begin {lem}
Let $u\in {\overline S[\varphi, \Omega]}$, we have
$$\int_{\Omega}G(\det\partial^2u)\, dx
=\lim_{h\rightarrow 0}\int_{\Omega}G(\det\partial^2u_h)\, dx.$$
\end {lem}

Finally, the existence of maximizers of $J$
in $\overline S[\varphi, \Omega]$ follows from
the following upper semi-continuity of the functional
$A(u)$ with respect to uniform convergence.

\begin {lem}
Suppose that ${u_n}\in{\overline S[\varphi, \Omega]}$
converge locally uniformly to $u$. Then
$$\limsup_{n\rightarrow\infty}\int_{\Omega}G(\det\partial^2u_n)\, dx
\leq\int_{\Omega}G(\det\partial^2u) \, dx.$$
\end {lem}

\begin {proof}
The proof is also inspired by \cite{Lu, TW1}, see also \cite{ZZ}.
Subtracting $G$ by the constant $G(0)$,
we may suppose that $G(0)=0$.
By Lemma 2.3, it suffices to prove it for
$u_n \in C^2(\overline\Omega)$ and we may assume that
$u_n$ converges uniformly to $u$ in $\overline \Omega$.

Denote by $S$ the supporting set of $\mu_s[u]$,
whose Lebesgue measure is zero.
By the upper semi-continuity of the Monge-Amp\`ere measure,
for any closed subset $F\subset\Omega\setminus S$,
\beq
\limsup_{n\rightarrow \infty}\int_F\det D^2u_n\, dx
\leq \int_F \det\partial^2u\, dx.
\eeq
For given $\epsilon,\epsilon'>0$, let
$$\Omega_k=\{x\in \Omega\setminus S\ |\
(k-1)\epsilon\leq \det \partial^2u<k\epsilon\},
\  k=0, 1, 2, ..., $$
and $\omega_k\subset \Omega_k$ be a closed set such that
$$|\Omega_k\backslash \omega_k|<\frac{\epsilon'}{2^{|k|}}.$$
For each $\omega_k$, by concavity of $G$ and
(2.10), we have
\beqs
\limsup_{n\rightarrow\infty}\frac{1}{|\omega_k|}
\int_{\omega_k} G(\det D^2u_n)\,dx
&\leq& \limsup_{n\rightarrow \infty}
G\left(\frac{\int_{\omega_k}\det D^2u_n\, dx}{|\omega_k|}\right)\\
&\leq& G\left(\frac{\int_{\omega_k}\det\partial^2u\, dx}{|\omega_k|}\right)\\
&\leq& G(k\epsilon).
\eeqs
It follows
\beqs
\limsup_{n\rightarrow \infty}\int_{\omega_k} G(\det D^2u_n)\, dx
&\leq& G(k\epsilon)|\omega_k|\\
&\leq& G((k-1)\epsilon)|\omega_k|+G(\epsilon)|\omega_k|\\
&\leq& \int_{\Omega_k}G(\det\partial^2u)\, dx+G(\epsilon)|\Omega_k|.
\eeqs
Hence,
$$\limsup_{n\rightarrow\infty}\int_{\bigcup\omega_k}G(\det D^2u_n)\, dx
\leq \int_{\Omega} G(\det \partial^2u)\,dx+G(\epsilon)|\Omega|. $$
By (2.9), letting $\epsilon$ go to $0$,
we can replace the domain of the left hand side integral by $\Omega$.
The lemma is proved.
\end {proof}

For the uniqueness of maximizers, we first prove a lemma.

\begin {lem}
For any maximizer $u$ of $J(\cdot)$, the
Monge-Amp\`ere measure $\mu[u]$ has no singular part.
\end {lem}

\begin{proof}
We use an argument from \cite{TW2} to prove the lemma.
 Suppose $\mu[u]$ has non-vanishing singular part
 $\mu_s[u]$. Then  for any $M>0$, there must exist a ball
$B_r\subset \Omega$ such that
 \beq
 \mu_s[u](B_r)\geq M(\mu_r[u](B_r)+|B_r|).
\eeq
 We consider  the following
Dirichlet problem for Monge-Amp\`ere operator,
$$ \begin{cases}
&\mu[v]=M\mu_r[u]+ M \ \text{in} \  B_r,\\
&v=u \  \text{on} \  \partial B_r.
\end{cases}$$
By the Alexander theorem, the above equation has  a unique
convex solution $v$. Note
\beq
\det\partial^2v=M\det\partial^2 u+ M,\ \text{in}\  B_r.
\eeq
 By comparison principle, $u\leq v$ in $B_r$, and  the set
$E=\{v>u\}$ is not empty. Define another convex  function
$\tilde u$ by
 $$\begin {cases}
&\tilde{u}= u \ \text{in}\  \Omega\setminus E,\\
&\tilde{u}=v \ \text{in}\  E.
\end {cases}$$
Then $\tilde u \in\overline S[\varphi, \Omega]$.
 We claim $J(\tilde{u})<J(u)$, so we get a contradiction to
the assumption that $u$ is a maximizer.
In fact, using (2.12),  we have
 \beqs
J(\tilde{u})-J(u)
&=&\int_E G(\det\partial^2v)\,dx-\int_E G(\det\partial^2u)\,dx
-\int_E f(v-u)\,dx\\
&=& \int_E G(\det \partial^2 u)-G(M(1+\det \partial^2 u))\,dx
-\int_E f(v-u)\,dx.
\eeqs
By the definition of $G$, the first integral goes to $-\infty$
as $M$ goes to $\infty$. The second integral is bounded
since $f$ is bounded.  The lemma is proved.
\end{proof}

In conclusion, we  have obtained the existence and uniqueness
of maximizers of $J$ in $\overline{S}[\varphi, \Omega]$.

\begin {theo}
Let $\Omega$ be a bounded, Lipschitz domain in $\mathbb R^n$.
Suppose $\varphi$ is a convex Lipschitz function defined
in a neighborhood of $\overline \Omega$ and $f\in{L^\infty}(\Omega)$.
There exists a unique function in
$\overline S[\varphi, \Omega]$ maximizing $J$.
\end {theo}

\begin {proof}
The existence follows from the upper semi-continuity of $A(u)$.
For the uniqueness, note that by the concavity of the
functional, if there exist two maximizers $u$ and $v$, then
$\partial^2u=\partial^2 v$ almost everywhere. Hence by Lemma 2.5 we have
$\mu[u]=\mu[v]$. By the uniqueness of generalized solutions to the
Dirichlet problem of the Monge-Amp\`ere equation, we conclude that $u=v$.
\end {proof}

\vskip 10pt

In Theorem 2.6, we only need the Lipschitz condition
on $\Omega$ and $\varphi$.
But later for the regularity,
we must assume the smoothness as stated in Theorem 1.1.
We point out again that the above argument
applies to the functional $J_0$,
and the existence and uniqueness of maximizers also hold for $J_0$.
But we will not study the maximizer of $J_0$ obtained in this way.

\vskip10pt

For our purpose of studying $J_0$,
we choose a sequence of functions
$G_k=G_{\delta_k}$ satisfying (a)-(c) with
$\delta_k \rightarrow 0$ as $k\rightarrow \infty$,
and consider the functionals
\beq
J_k(u)=A_k(u)-\int_{\Omega}fu\, dx,
\eeq
where
\beq
A_k(u)=\int_\Omega G_k(\det D^2u)\, dx.
\eeq
By Theorem 2.6, there exists $u^{(k)}\in{\overline S[\varphi, \Omega]}$
maximizing the functional $J_k$ in $\overline S[\varphi, \Omega]$.
It is clear that $u^{(k)}$ converges to a convex function $u_0$
in $\overline S[\varphi, \Omega]$.
We will prove that in dimension $2$, $u_0$ solves the problem (1.7).
The main point is to prove the smoothness of $u_0$.
Once we have the regularity of $u_0$, the uniqueness follows
immediately by the concavity of $A_0$ and
the uniqueness of generalized solutions to the
Dirichlet problem of the Monge-Amp\`ere equation.
Hence, In the rest of this paper,
we prove that $u_0$ is smooth in $\Omega$
and satisfies Abreu's equation.

\vskip 10pt

\section{Interior estimates}

\vskip 10pt

In this section, we establish the interior estimates for  equation (2.6).

\begin{lem}
Let $u$ be a convex smooth solution to
(2.6)  in a convex domain $\Omega$. Assume that $u<0$ in $\Omega$ and
$u=0$ on $\p \Omega$. Then there is a positive constant $C$
depending only on $n$, $\sup|\nabla u|$,
$\sup|u|$, $\sup|f|$ and independent of
$\delta$,  such that
$$(-u)^{n}\det D^2u \leq C.$$
\end{lem}

\begin{proof}
Let
$$z=-\log d- \log{(-u)^\beta}-|\nabla u|^2,$$
where $\beta$ is a positive number to be determined later.
Then $z$ attains its minimum at a point $p$ in $\Omega$.
We may assume that $d(p)> \delta$ so that
$w=d^{-1}$ in a small neighborhood of $p$.
Otherwise,  the estimate follows directly.
Hence, at $p$, it holds
$$z_i=0,\  u^{ij}z_{ij}\geq 0.$$
We can rewrite $z$ as
$$z=\log w- \log{(-u)^\beta}-|\nabla u|^2$$ near $p$.
By computation,
\beqn
&&z_i=\frac{w_i}{w}-\frac{\beta u_i}{u}-2u_{ki}u_k,\\
&&z_{ij}=\frac{w_{ij}}{w}-\frac{w_iw_j}{w^2}-\frac{\beta u_{ij}}{u}
+\frac{\beta u_iu_j}{u^2}
-2u_{kij}u_{k}-2u_{ki}u_{kj}.
\eeqn
On the other hand, since $\det D^2u=w^{-1}$,
$$u^{ab}u_{abi}=(-\log{w})_i=-\frac{w_i}{w}.$$
Therefore we have
$$u^{ij}z_{ij}= u^{ij}\frac{w_{ij}}{w}
-u^{ij}\frac{w_iw_j}{w^2} -\frac{\beta n}{u}
+\beta\frac{u^{ij}u_iu_j}{u^2}+2\frac{w_k}{w}u_{k}-2\triangle u.$$
By (3.1),
\begin{eqnarray*}
u^{ij}\frac{w_iw_j}{w^2}
&=& \beta^2 u^{ij}\frac{ u_iu_j}{u^2}+\frac{4\beta|Du|^2 }{u}+4u_{ij}u_iu_j,\\
 \frac{w_k}{w}u_{k}&=&\frac{\beta|Du|^2 }{u}+2u_{ij}u_iu_j.
\end{eqnarray*}
It follows
$$u^{ij}z_{ij}=f-\frac{\beta n}{u}-2\bigtriangleup u
-\frac{2\beta|Du|^2 }{u}-
(\beta^2-\beta)u^{ij}\frac{ u_iu_j}{u^2}\geq 0.$$
Choosing $\beta=n$, we have
$$ (-u)[\det D^2u]^{\frac{1}{n}} \leq (-u)\bigtriangleup u  \leq C $$
at $p$. The lemma follows.
\end{proof}

For the lower bound estimate of the determinant,
we consider the Legendre function $u^*$ of $u$.
If $u$ is smooth,
$u^*$ is defined on $\Omega^*=Du(\Omega)$,
given by
$$u^*(y)=x\cdot y-u(x),$$
where $x$ is the point determined by $y=Du(x)$.
Differentiating $y=Du(x)$, we have
$$\det D^2u(x)=[\det D^2u^*(y)]^{-1}.$$
The dual functional with respect to the Legendre function is given by
$$J^*(u^*)=A^*(u^*)
-\int_{\Omega^*}f(Du^*)(yDu^*-u^*)\det D^2u^*\, dy,$$
where
$$A^*(u^*)=\int_{\Omega^*}G([\det D^2u^*]^{-1})\det D^2u^*\, dy.$$
If  $u$ is a solution to equation (2.6) in $\Omega$,
it is a local maximizer of the functional $J$.
Hence $u^*$ is a critical point of $J^*$ under local perturbation,
so it satisfies the Euler equation
of the dual functional $J^*$, namely in $\Omega^*$
\beq
{u^*}^{ij}w^*_{ij}=-f(Du^*),
\eeq
where
\beq
w^*=G({d^*}^{-1})-{d^*}^{-1}G'({d^*}^{-1}), \  d^*=\det D^2 u^*.
\eeq
Note that on the left hand side of (3.3), it is ${u^*}^{ij}$,
the inverse of $(u^*_{ij})$.

\begin{lem}
Let $u^*$ be a smooth convex solution to (3.3) in $\Omega^*$ in dimension $2$.
Assume that $u^*<0$ in $\Omega^*$ and $u^*=0$ on $\p \Omega^*$.
Then there is a positive constant $C$ depending only on $\sup|\nabla u^*|$,
$\sup |u^*|$, $\inf f$ and independent of
$\delta$ such that
$$(-u^*)^{2}\det D^2u^*\leq C.$$
\end{lem}

\begin{proof}
We consider
$$z=-\log d^*-\log(-u^*)^\beta-\alpha|\nabla u^*|^2,$$
where $\alpha$, $\beta$ are positive numbers to be determined below.
Since $z$ tends to $\infty$ on $\partial \Omega^*$, it must attain its minimum
at some point $p\in{\Omega^*}$. At $p$ we have
$$z_i=0, \  {u^*}^{ij}z_{ij}\geqslant 0.$$
By (3.4), we compute
\beqn
w^*_i&=&G''({d^*}^{-1}){d^*}^{-3}d^*_i,\\
w^*_{ij}&=&-G'''({d^*}^{-1}){d^*}^{-5}d^*_id^*_j
-3G''({d^*}^{-1}){d^*}^{-4}d^*_id^*_j
+G''({d^*}^{-1}){d^*}^{-3}d^*_{ij}.
\eeqn
On the other hand, by computation,
\beqn
&&z_i=-\frac{d^*_i}{d^*}-\beta\frac{u^*_i}{u^*}-2\alpha u^*_{ki}u^*_k,\\
&&z_{ij}=-\frac{d^*_{ij}}{d^*}
+\frac{d^*_i d^*_j}{{d^*}^2}
-\beta\frac{u^*_{ij}}{u^*}+\beta\frac{u^*_iu^*_j}{{u^*}^2}
-2\alpha u^*_{kij}u^*_k-2\alpha u^*_{ki}u^*_{kj}.
\eeqn
It follows
$$ {u^*}^{ij}z_{ij}
=-\frac{{u^*}^{ij}d^*_{ij}}{d^*}+
\frac{{u^*}^{ij}d^*_i d^*_j}{{d^*}^2}
-\beta\frac{n}{u^*}+\beta\frac{{u^*}^{ij}u^*_i u^*_j}{{u^*}^2}
-2\alpha \frac{d^*_k}{d^*}u^*_k-2\alpha\bigtriangleup u^*.$$
By (3.6) and equation (3.3), we have
$$\frac{{u^*}^{ij}d^*_{ij}}{d^*}
=-\frac{{d^*}^2}{G''({d^*}^{-1})}f
+\frac{{d^*}^{-1}G'''({d^*}^{-1})}{G''({d^*}^{-1})}\frac{{u^*}^{ij}d^*_i d^*_j}{{d^*}^2}
+3\frac{{u^*}^{ij}d^*_i d^*_j}{{d^*}^2}.$$
We may assume that $f(p)<0$.
By condition (b) for $G$,
$$\frac{{d^*}^2}{G''({d^*}^{-1})}\leq -C_1^{-1},\  \
\left|\frac{{d^*}^{-1}G'''({d^*}^{-1})}{G''({d^*}^{-1})}\right|\leq C_2.$$
Hence,
$$\frac{{u^*}^{ij}d^*_{ij}}{d^*}\geq
C_1^{-1}\inf f+(3-C_2)\frac{{u^*}^{ij}d^*_i d^*_j}{{d^*}^2}.$$
So we have
$${u^*}^{ij}z_{ij}
\geq-C_1^{-1}\inf f+(C_2-2)\frac{{u^*}^{ij}d^*_i d^*_j}{{d^*}^2}
-\frac{\beta n}{u^*}+\beta\frac{{u^*}^{ij}u^*_i u^*_j}{{u^*}^2}
-2\alpha \frac{d^*_k}{d^*}u^*_k-2\alpha\bigtriangleup u^*.$$
By (3.7),
\beqs
\frac{{u^*}^{ij}d^*_i d^*_j}{{d^*}^2}&=&
\beta^2\frac{{u^*}^{ij}u^*_iu^*_j}{{u^*}^2}
+4\alpha \beta\frac{|\nabla u^*|^2}{u^*}+4\alpha^2 u^*_{lk}u^*_lu^*_k,\\
\frac{d^*_k}{d^*}u^*_k&=&
-\beta\frac{|\nabla u^*|^2}{u^*}-2\alpha u^*_{lk}u^*_lu^*_k.
\eeqs
Therefore
\beqs
&&-C_1^{-1}\inf f+[\beta+(C_2-2)\beta^2]\frac{{u^*}^{ij}u^*_iu^*_j}{{u^*}^2}
-\frac{n\beta}{u^*}\\
&&\ \ \ +[4(C_2-2)+2]\alpha \beta\frac{|\nabla u^*|^2}{u^*}
+[4(C_2-2)+4]\alpha^2 u^*_{lk}u^*_lu^*_k
-2\alpha\bigtriangleup u^*\geq 0.
\eeqs
Choose $\alpha$ small enough 
depending on $\sup|\nabla u^*|$ such that
$$[4(C_2-2)+4]\alpha^2 u^*_{lk}u^*_lu^*_k\leq \alpha\bigtriangleup u^*.$$
Using the fact
${u^*}^{11}+{u^*}^{22} =\frac{\bigtriangleup u^*}{\det D^2u^*}$ in dimension 2,
we have
$$\frac{{u^*}^{ij}u^*_iu^*_j}{{u^*}^2}\leq
\frac{|\nabla u^*|^2}{{u^*}^2}\frac{\bigtriangleup u^*}{\det D^2u^*}.$$
It follows
$$-C_1^{-1}\inf f
+C' \frac{|\nabla u^*|^2}{{u^*}^2}\frac{\bigtriangleup u^*}{\det D^2u^*}
-\frac{\beta n}{u^*} +C''\frac{|\nabla u^*|^2}{u^*}
-\alpha\bigtriangleup u^*\geq 0,$$
where $C'$, $C''$ are constants depending only on
$\alpha$, $\beta$, $C_1$ and $C_2$.
If $$\frac{\alpha}{2}\bigtriangleup u^*
-C'\frac{|\nabla u^*|^2}{{u^*}^2}\frac{\bigtriangleup u^*}{\det D^2u^*}
\leq 0,$$
we obtain
$$(-u^*)^2\det D^2u^*\leqslant C$$
at $p$. Otherwise, we have
$$-C_1^{-1}\inf f-\frac{\beta n}{u^*} +C''\frac{|\nabla u^*|^2}{u^*}
-\frac{\alpha}{2}\bigtriangleup u^*\geq 0.$$
Hence, we also obtain
$$(-u^*)^2\det D^2u^*\leq (\bigtriangleup u^*)^2(-u^*)^2\leqslant C$$
at $p$. The lemma follows by  choosing $\beta=n=2$.
\end{proof}

\begin{rem}\ \newline
(i) The determinant estimates above is independent of $\delta$.
This leads us to use the approximation $\{G_k\}$;
\newline
(ii) The estimate depends only on $\inf f$. This is crucial in Section 7;
\newline
(iii) In Lemma 3.2, the estimate only holds in dimension 2.
Since if we do not have the relation
${u^*}^{11}+{u^*}^{22} =\frac{\bigtriangleup u^*}{\det D^2u^*}$,
we can not deal with the term
$\frac{{u^*}^{ij}u^*_iu^*_j}{{u^*}^2}$ in the proof.
This is why we can not extend Theorem 1.1 to higher dimensions.
\end{rem}

\vskip 10pt

To apply the above determinant estimates, we first introduce
the {\it modulus of convexity} for convex functions.
The modulus of convexity of $u$ at $x$ is defined by
\beq
h_{u, x}(r)=\sup \{\delta\geq 0\ |\ S_{\delta, u}(x)\subset B_r(x)\},\ r>0
\eeq
and the modulus of convexity of $u$ on $\Omega$ is defined by
\beq
h_{u, \Omega}(r)=\inf_{x\in \Omega}h_{u, x}(r),
\eeq
where
$$S_{\delta, u}(x)=\{y\in\Omega\ |\ u(y)<\delta+a_x(y)\}$$
and $a_x$ is a tangent plane of $u$ at $x$.
When no confusions arise, we will also write $S_{\delta, u}(x)$ as
$S_{\delta, u}$ or $S_\delta$, for brevity.

\begin{lem}
Let $u\in{C^4(\Omega)}$ be a locally uniformly convex solution 
to (2.6) in dimension 2.

(i) Assume $f\in{L^\infty(\Omega)}$. Then
$$\|u\|_{W^{4, p}(\Omega')}\leq C$$
for any $p>1$ and $\Omega'\subset\subset\Omega$,
where $C$ depends on $n$, $p$, $\sup |f|$, $dist(\Omega', \partial \Omega)$
and the modulus of convexity of $u$.

(ii) Assume $f\in{C^\alpha(\Omega)}$. Then
$$\|u\|_{C^{4, \alpha}(\Omega')}\leq C$$
for any $\alpha\in{(0,1)}$ and $\Omega'\subset\subset\Omega$,
where $C$ depends on $n$, $\alpha$, $\sup |f|$, 
$dist(\Omega', \partial \Omega)$
and the modulus of convexity of $u$.
\end{lem}

\begin{proof}
For any $x\in\Omega$, by Lemma 3.1, we have
$$\det D^2u(x)\leq C$$
where $C$ is a constant depending only on $f$, 
$\delta=dist(x, \partial \Omega)$ and $h_{u, \Omega}$.
Let $y=Du(x)\in \Omega^*$. By (3.9), (3.10), we have
$$S_{\delta^*, u^*}(y)\subset \Omega^*,$$
where $\delta^*=h_{u,\Omega}(\frac{\delta}{2})$.
Furthermore, since $|Du^*|\leq diam (\Omega)$,
we also have
$$dist(y,\partial \Omega^*)\geq \frac{\delta^*}{2 diam(\Omega)}.$$
Hence, by Lemma 3.2,
$$\det D^2u(x)=[\det D^2u^*(y)]^{-1}\geq C',$$
where $C'$ is a constant depending only on $f$,
$\delta$ and $h_{u, \Omega}$.

Once the determinant $\det D^2 u$ is bounded,
we also have the Holder continuity of $\det D^2 u$
by Caffarelli-Gutierrez's H\"older continuity
for linearized Monge-Amp\`ere equation \cite{CG}.
Then we have the $W^{2, p}$ and $C^{2,\alpha}$ regularity for $u$ by
Caffarelli's  $W^{2, p}$ and $C^{2,\alpha}$ estimates
for Monge-Amp\`ere equation \cite{Caf1, JW}, respectively.
Higher regularity then follows from the standard elliptic regularity theory 
\cite{GT}.
\end{proof}

We will estimate in Section 6 and 7 the modulus of convexity for the solution
$u$ in dimension 2. In Section 4 we consider the change of equation (2.6)
under a coordinate transformation and establish the a priori estimates
for the equation after the transformation.

\vskip 10pt

\section{Equations after rotations in $\mathbb R^{n+1}$}

\vskip 10pt

Equation (2.6) is invariant under transformations
of the $x$-coordinates in $\mathbb{R}^n$, but it changes
when taking transformations in $\mathbb{R}^{n+1}$.
We note that the affine maximal surface equation is invariant
under uni-modular transformations in $\mathbb{R}^{n+1}$,
which plays an important part \cite{TW1}.
In order to establish the estimate of the modulus of convexity,
we also need to consider the equation under rotations in $\mathbb R^{n+1}$.
In this section we will derive the new equation
under a rotation in $\mathbb R^{n+1}$ and
establish the a priori estimates for it.

For our purpose it suffices to consider the rotation
$z=T x$, given by
\begin{eqnarray}
&&z_1=-x_{n+1}, \\
&& z_2=x_2,\  ...,\  z_n=x_n, \\
&&z_{n+1}=x_1,
\end{eqnarray}
which  fixes $x_2, ..., x_n$ axes.
Assume that the graph of $u$,
$\mathcal M=\{(x, u(x))\in{\mathbb R^{n+1}}  \  |  \  x\in\Omega\}$,
can be represented by a convex function
$z_{n+1}=v(z_1, ...,  z_n)$ in $z$-coordinates, in a domain $\hat \Omega$.
To derive the equation for $v$, we compute the change of the functional $A_0$.

\beqn
A(u)&=&\int_\Omega G(\det D^2u)\,dx\\
&=&\int_\Omega G\left(\frac{\det D^2u}{(1+|Du|^2)^{\frac{n+2}{2}}}
(1+|Du|^2)^{\frac{n+2}{2}}\right) \,dx\nonumber\\
&=&\int_{\mathcal M} G\left(K(1+|Du|^2)^{\frac{n+2}{2}}\right)
(1+|Du|^2)^{-\frac{1}{2}} \,d\Sigma, \nonumber
\eeqn
where $K$ is the Gaussian curvature of $\mathcal M$
and $d\Sigma$ the volume element of the hypersurface.
It is easy to verify that
\beq
u_1=-\frac{1}{v_1}, \ u_2=\frac{v_2}{v_1}, \ ...,
\ u_n=\frac{v_n}{v_1},
\eeq
where $v_i=\frac{\partial v}{\partial z_i}$. So we have
$$1+|Du|^2=\frac{1+|Dv|^2}{v_1^2}.$$
Hence we obtain
\beq
A(u)
=\int_{\hat \Omega} G(v_1^{-(n+2)}\det D^2 v) (v_1^2)^{\frac{1}{2}}\,dz
:=\hat A(v).
\eeq
In addition,
\beqs
\int_\Omega f(x)u(x) dx
&=&\int_{\mathcal M} f\cdot u\cdot (1+|Du|^2)^{-\frac{1}{2}} \,d\Sigma\\
&=&\int_{\hat \Omega}f(v, z_2, ..., z_n)\cdot (-z_1)\cdot (v_1^2)^{\frac{1}{2}}\,dz.
\eeqs
Let
$$\hat J(v)=\hat A(v)
-\int_{\hat \Omega}f(v, z_2, ..., z_n)\cdot (-z_1)\cdot (v_1^2)^{\frac{1}{2}}\,dz.$$
After computing the Euler equation for the functional $\hat J(v)$, we have

\begin{lem}
Let $u$ be a solution of (2.6). Let $T$ and $v$ be as above.
Then $v$ satisfies the equation
\beq
V^{ij}(d^{-1})_{ij}=g-f_1z_1v_1+f_1z_1+f
\eeq
in the set $\{z  \ | \ v_1^{-(n+2)}d> \delta \}$,
where $(V^{ij})$ is the cofactor matrix of $(v_{ij})$, $d=\det D^2v$ and
\beqs
g&=&2v^{kl}v_{kl1} \frac{1}{v_1}-(n+2) \frac{v_{11}}{v_1^{2}},\\
f&=&f(v, z_2, ..., z_n),\\
f_1&=&\frac{\partial f}{\partial x_1}(v, z_2, ..., z_n).
\eeqs
\end{lem}

\begin{rem}
In the proof of strict convexity in Section 6,
we will use the upper bound estimate for $\det D^2v$ given below.
Since the lower bound for $\det D^2v$ will not be used,
we do not need the explicit form of the equation for $v$
outside the set $\{z  \ | \ v_1^{-(n+2)}d> \delta \}$.
Therefore in (4.7), we calculate the Euler equation only
in the set $\{z  \ | \ v_1^{-(n+2)}d> \delta \}$.
\end{rem}

Next we prove a determinant estimate for $v$.
Assume $v$ satisfies
\beq
\begin{array}{l}
v\geq 0,\ v\ge z_1, \  v_1\geq 0, \ \ \text{and $v(0)$\ \
is as small as we want such that }\ \ \\
\text{for the positive constants $\epsilon$ and $c$   in $(0, \frac 12)$,
$\hat\Omega_{\epsilon, c}$ is a nonempty open set},
\end{array}
\eeq
where
$$\hat v=v-\epsilon z_1-c \  \text{and} \
\  \hat\Omega_{\epsilon, c}=\{z\ | \  \hat v(z)< 0\}.$$
Then $\hat v$ satisfies
\beq
\hat V^{ij}(\hat d^{-1})_{ij}
=\hat g-\hat f_1z_1(\hat v_1+\epsilon)+\hat f_1z_1+\hat f
\eeq
in the set $\{z  \ | \ (\hat{v}_1+\epsilon)^{-(n+2)}\hat{d}
> \delta \}\cap \hat\Omega_{\epsilon, c}$,
where $\hat d=\det D^2\hat v$ and
\beqn
\hat g&=&2\hat v^{kl}\hat v_{kl1}\frac{1}{\hat v_1+\epsilon}
-(n+2)\frac{\hat v_{11}}{(\hat v_1+\epsilon)^2},\\
\hat f&=&f(\hat v+\epsilon z_1+c, z_2, ..., z_n),\\
\hat f_1&=&\frac{\partial f}{\partial x_1}(\hat v+\epsilon z_1+c, z_2, ..., z_n).
\eeqn

\begin{lem}
Let $\hat v$ be as above.
Then there exists $C>0$
depending only on $\sup |\hat f|$, $\sup|\nabla \hat f|$,
$\sup_{\hat\Omega_{\epsilon, c}} |\hat v|$
and $\sup_{\hat\Omega_{\epsilon, c}}|D\hat v|$, 
but independent of $\delta$,
such that
$$(-\hat v)^n \det D^2\hat v\leq C.$$
\end{lem}

\begin{proof}
Consider
$$\eta=\log {w}-\beta\log{(-\hat v)}-A|D\hat v|^2,$$
where $w=\hat d^{-1}$, and $\beta$, $A$ are positive numbers
to be determined below.
Then $\eta$ attains its minimum at a point $p$ in $\hat\Omega_{\epsilon, c}$. 
Hence, at $p$, it holds
$$\eta_i=0,\  \hat v^{ij}\eta_{ij}\geq 0.$$
We can suppose that 
$p\in\{z  \ | \ (\hat v_1+\epsilon)^{-(n+2)}\hat d> \delta \}$.
Otherwise, we have
$$(\hat v_1+\epsilon)^{-(n+2)}\hat d\leq \delta$$
and then the estimate follows. By computation,
\beqn
\eta_i&=&\frac{w_i}{w}-\frac{\beta \hat v_i}{\hat v}-2A\hat v_{ki}\hat v_k,\\
\eta_{ij}&=&\frac{w_{ij}}{w}-\frac{w_iw_j}{w^2}-\frac{\beta \hat v_{ij}}{\hat v}
+\frac{\beta \hat v_i \hat v_j}{\hat v^2}
 -2A\hat v_{kij}\hat v_{k}-2A\hat v_{ki}\hat v_{kj},\\
\frac{w_k}{w}&=&-\hat v^{ij}\hat v_{ijk}.
\eeqn
By  (4.15),
$$\hat g=-2\frac{w_1}{w}\frac{1}{\hat v_1+\epsilon}
-(n+2)\frac{\hat v_{11}}{(\hat v_1+\epsilon)^2}.$$
Therefore we have
\beqs
\hat v^{ij}\eta_{ij}&=&-\frac{\hat v^{ij}w_iw_j}{w^2}
-\frac{w_1}{w}\frac{2}{\hat v_1+\epsilon}-\frac{\beta n}{\hat v}
-(n+2)\frac{\hat v_{11}}{(\hat v_1+\epsilon)^2}
+\frac{\beta \hat v^{ij}\hat v_i\hat v_j}{\hat v^2}+2A\frac{w_k}{w}\hat v_{k}\\
&&\  \  -2A\triangle \hat v-\hat f_1z_1(\hat v_1+\epsilon)+\hat f_1z_1+\hat f.
\eeqs
By (4.13),
\beqs
\frac{\hat v^{ij}w_iw_j}{w^2}&=&
\beta^2\hat v^{ij}\frac{ \hat v_i\hat v_j}{\hat v^2}
+4A^2\hat v_{ij}\hat v_i\hat v_j+4A\beta\frac{|D\hat v|^2}{\hat v},\\
\frac{w_1}{w}\frac{2}{\hat v_1+\epsilon}
&=&\frac{2\beta \hat v_1}{(\hat v_1+\epsilon)\hat v}
+4A\frac{\hat v_{1k}\hat v_k}{\hat v_1+\epsilon},\\
\frac{w_k}{w}\hat v_{k}
&=&\beta\frac{|D\hat v|^2}{\hat v}+2A\hat v_{ij}\hat v_i\hat v_j.
\eeqs
Hence, we have
\begin{eqnarray}
\hat v^{ij}\eta_{ij}&=&-(n+2)\frac{\hat v_{11}}{(\hat v_1+\epsilon)^2}
-4A\left(\frac{\hat v_{11}\hat v_1}{\hat v_1+\epsilon}
+\sum_{k=2}^{n}\frac{\hat v_{1k}\hat v_k}{\hat v_1+\epsilon}\right)
-\frac{2\beta \hat v_1}{(\hat v_1+\epsilon)\hat v}-2A\bigtriangleup \hat v
\nonumber\\
&&\quad  -\frac{\beta n}{\hat v}-2A\beta\frac{|D\hat v|^2}{\hat v}
-(\beta^2-\beta)\hat v^{ij}\frac{\hat v_i\hat v_j}{\hat v^2}
-\hat f_1z_1(\hat v_1+\epsilon)+\hat f_1z_1+\hat f.
\end{eqnarray}
We choose $\beta>1$ such that $\beta^2-\beta>0$.
By the positive definiteness of $\hat v_{ij}$,
it holds $\hat v_{1k}^2\leq \hat v_{11}\hat v_{kk}$ for any $k=2, ..., n$, so
there is $C'$ depending on $n$ and $|D\hat v|$, such that
\beq
\sum_{k=2}^{n}\frac{|\hat v_{1k}\hat v_k|}{\hat v_1+\epsilon}
\leq  \frac{1}{4}\sum_{k=2}^{n}\hat v_{kk}
+C'\frac{\hat v_{11}}{(\hat v_1+\epsilon)^2}
\leq\frac{1}{4} \bigtriangleup \hat v
+C'\frac{\hat v_{11}}{(\hat v_1+\epsilon)^2}.
\eeq
It follows
\beqn
&&-\frac{(n+2-4AC')\hat v_{11}}{(\hat v_1+\epsilon)^2}
-4A\frac{\hat v_{11}\hat v_1}{\hat v_1+\epsilon}
-\frac{2\beta \hat v_1}{(\hat v_1+\epsilon)\hat v}
-A\bigtriangleup \hat v-\frac{\beta n}{\hat v}\nonumber\\
&&\  \  \  \  \  \  \  \  \  \  \  \ -2A\beta\frac{|D\hat v|^2}{\hat v}
-\hat f_1z_1(\hat v_1+\epsilon)+\hat f_1z_1+\hat f
\geq 0.
\eeqn
Choosing $A$ small enough such that $n+2-4AC'>0$.
Then by a Schwarz inequality, there exists a $C_0>0$ depending
only on $|D \hat v|$ such that
\beq
-\frac{(n+2-4AC')\hat v_{11}}{(\hat v_1+\epsilon)^2}
-4A\frac{\hat v_{11}\hat v_1}{\hat v_1+\epsilon}\leq C_0A^2\hat v_{11}.
\eeq
By (4.18), (4.19),  we have
$$0\leq C_0A^2\hat v_{11}
-\frac{2\beta \hat v_1}{(\hat v_1+\epsilon)\hat v}
-\frac{\beta n}{\hat v}-A\bigtriangleup \hat v
-2A\beta\frac{|D\hat v|^2}{\hat v}
-\hat f_1z_1(\hat v_1+\epsilon)+\hat f_1z_1+\hat f.$$
Choosing $A$ small enough furthermore
such that $C_0A^2\leq\frac{A}{2}$, 
and observing that 
$$\frac{2\beta \hat v_1}{(\hat v_1+\epsilon)\hat v}
=\frac{2\beta}{\hat v}-\frac{2\beta \epsilon}{(\hat v_1+\epsilon)\hat v}
\geq \frac{2\beta}{\hat v},$$
we have
$$-\frac{\beta (n+2)}{\hat v}-\frac{A}{2}\bigtriangleup \hat v
-2A\beta\frac{|D\hat v|^2}{\hat v}
-\hat f_1z_1(\hat v_1+\epsilon)+\hat f_1z_1+\hat f
\geq 0,$$
which implies
$$(-\hat v)\triangle \hat v\leq C$$
at $p$. Hence, choosing $\beta=n$, the lemma follows by
$$e^{\eta(x)}\geq e^{\eta(p)}=\hat d^{-1}(-\hat v)^{-n}e^{-A |D\hat v|^2}
\geq [\frac{(-\hat v)\triangle \hat v}{n}]^{-n}
e^{-A |D\hat v|^2}\geq C.$$
\end{proof}

\vskip 10pt

\section{Approximation}

\vskip 10pt

We will use a penalty method and solutions to the second boundary
value problem to construct a sequence of smooth convex solutions to (2.6)
to approximate the maximizer of $J(u)$.
This section is similar to \S6 in \cite{TW2}.

First, we consider a second boundary value problem
with special non-homogenous
term $f$. Let $B=B_R(0)$ be a ball with $\Omega\subset\subset B$
and $\varphi\in{C^2(\overline B)}$ be a uniformly convex function in $B$
vanishing on $\partial B$.
Suppose $H$ is a nonnegative smooth function
defined in the interval $(-1, 1)$ such that
\beq
H(t)=\begin{cases}
(1-t)^{-2n},  &\  t\in{(\frac{1}{2}, 1)},\\
(1+t)^{-2n},  &\  t\in{(-1,\frac{1}{2})}.\\
\end{cases}
\eeq
Extend the function $f$ to $B$ such that
$$f(x, u)=\begin{cases}
f(x)\ &\text{if}\  x\in{\Omega},\\
h(u-\varphi(x))  \  &\text{if}\ x\in{B\setminus \Omega},
\end{cases}$$
where $h(t)=H'(t)$.

\begin{lem}
Let $f(x, u)$ be as above.
Suppose $\partial \Omega$ is Lipschitz continuous.
Then there exists a locally uniformly convex solution
to the second boundary problem
\begin{eqnarray}
U^{ij}w_{ij}&=&f(x, u)\ \ \text{in}\  B,\\
w&=&G'(d), \ \text{in}\ B,\nonumber\\
u&=&\varphi  \ \  \text{on}\ \partial B,\nonumber\\
w&=&1  \  \  \text{on}\  \partial B\nonumber
\end{eqnarray}
with $u\in{W^{4, p}_{loc}(B)\cap C^{0,1}(\overline B)}$, for all $p<\infty$,
and $w\in{C^0(\overline\Omega)}$.
\end{lem}

\begin{proof}
By the discussion of the second boundary problem in the Appendix,
it suffices to prove that for any solution $u$ to (5.2),
$|f(x, u)|\leq C$ for some constant $C$ independent of $u$.
Note that by our choice of $H$, a solution to (5.2)
is bounded from below.

\vskip 8pt

First, we prove an estimate of the determinant near the
boundary $\partial B$. By the definition of $H$
and the convexity of $u$,
$f$ is bounded from above near $\partial B$.
For any boundary point $x_0\in{\partial B}$, we suppose by a rotation of
axes that $x_0=(R, 0, ..., 0)$. There exists $\delta_0>0$ independent of $x_0$
such that $f$ is bounded from above in $B\cap\{x_1>R-\delta_0\}$.
Choose a linear function $l=ax_1+b$ such that
$l(x_0)<u(x_0)=0$ and $l>u$ on $x_1=R-\delta_0$.
Let
$$z=w+\log w-\beta\log(u-l),$$
where $\beta>0$ is to be determined below.
If $z$ attains its minimum at a boundary point on $\partial B$,
by the boundary condition $w=1$, $z\geq -C$ near $\partial B$.
If $z$ attains its minimum at a interior point $y_0\in{\{u>l\}}$,
we have, at $y_0$,
\begin{eqnarray}
0=z_i&=&w_i+\frac{w_i}{w}-\beta\frac{(u-l)_i}{u-l},\\
z_{ij}&=&w_{ij}+\frac{w_{ij}}{w}-\frac{w_iw_j}{w^2}
-\beta\frac{(u-l)_{ij}}{u-l}+\beta\frac{(u-l)_i(u-l)_j}{(u-l)^2}.
\end{eqnarray}
By (5.3),
$$\frac{w_i}{w}=\frac{\beta}{1+w}\frac{(u-l)_i}{u-l}.$$
It follows by (5.4) and equation (5.2)
$$0\leq u^{ij}z_{ij}=\frac{f}{d}+\frac{f}{dw}-\frac{\beta n}{u-l}+
\left[\beta-\frac{\beta^2}{(1+w)^2}\right]\frac{u^{ij}(u-l)_i(u-l)_j}{(u-l)^2}.$$
We may suppose that $w\leq 1$.
Choose $\beta$ large enough such that
$$\beta-\frac{\beta^2}{(1+w)^2}\leq 0.$$
So we have $w(y_0)\geq C$.
Therefore,  $\det D^2u\leq C$ near $\partial B$.

By the above determinant estimate near $\partial B$,
it follows that $|Du|$ is bounded near $\partial B$.
By the convexity of $u$,
$$\sup_B|Du|\leq C.$$

\vskip 8pt

Next, we prove that $f$ is bounded from below.
We note that by the Lipschitz continuity of $\partial \Omega$,
there exists positive constants
$r, \kappa$ such that for any $p\in{B\setminus\Omega}$,
there is a unit vector $\gamma$ such that the round cone
$\mathcal C_{p, \gamma, r, \kappa}\subset{B\setminus \Omega}$,
where
$$\mathcal C_{p, \gamma, r, \kappa}
:=\{x\in{\mathbb R^n} \ |\
|x-p|<r,\  \langle x-p, \gamma\rangle>\cos\kappa\}.$$
Assume that $M=-\inf _B f$ is attained at $x_0\in B$. If $x_0\in \Omega$,
then $M=\|f\|_{L^\infty(\Omega)}$. If $x_0\in{B\setminus\Omega}$,
we have $$M=2n[1+u(x_0)-\varphi(x_0)]^{-2n-1},$$
that is,
$$u(x_0)-\varphi(x_0)=\left(\frac{M}{2n}\right)^{-\frac{1}{2n+1}}-1.$$
Let $l_0$ be the tangent plane of $\varphi$ at $x_0$.
Since we have the gradient estimate of $u$,
there exists a uniform $\delta_0$ such that
$$0\leq 1+u(x)-\varphi(x)\leq 2\left(\frac{M}{2n}\right)^{-\frac{1}{2n+1}}$$
and
$$0\leq 1+u(x)-l_0(x)\leq 2\left(\frac{M}{2n}\right)^{-\frac{1}{2n+1}}$$
in the cone $\mathcal C_{x_0, \gamma,
\delta_0\left(\frac{M}{2n}\right)^{-\frac{1}{2n+1}}, \kappa}$.
Let $\omega_0=\{x\ |\  u(x)<l_0(x)\}$.
It is clear that when $M$ is sufficiently large,
$$\mathcal C_{x_0, \gamma, \delta_0\left(\frac{M}{2n}\right)^{-\frac{1}{2n+1}}, \kappa}\subset\omega_0.$$
Integrating by parts, we have
\begin{eqnarray*}
\int_{\omega_0}U^{ij}w_{ij}(u-l_0)\,dx
&=&-\int_{\omega_0}U^{ij}w_{j}(u-l_0)_i\,dx\nonumber\\
&=&-\int_{\partial\omega_0}wU^{ij}(u-l_0)_i\gamma_j\, dS
+\int_{\omega_0}w\det D^2u \, dx,
\end{eqnarray*}
where $dS$ is the volume element of $\partial \omega_0$.
$u-l_0$ vanishes on the boundary,
so $U^{ij}(u-l_0)_i\gamma_j\geq 0$.
The first integral on the right-hand side is negative.
Hence, we obtain
\beq
\int_{\omega_0}f(x, u)(u-l_0)\,dx
\leq \int_{\omega_0}w\det D^2u\, dx\leq C.
\eeq
Note that the last inequality follows by the condition
$\lim_{t\rightarrow 0}tF'(t)\leq C_3$
in the assumption (c) on $G$.
Estimating the integral in the cone, we have
\beq
\int_{\omega_0}f(x, u)(u-l_0)\,dx
\geq 2^{-2n-1}M
\cdot\left[1-2\left(\frac{M}{2n}\right)^{-\frac{1}{2n+1}}\right]
\cdot C \cdot \left(\frac{M}{2n}\right)^{-\frac{n}{2n+1}}.
\eeq
Therefore $M\leq C$ follows from (5.5), (5.6).
\vskip 8pt

Finally, we prove  that $f$ is bounded from above.
For any $\delta>0$, let
$$\Omega_\delta=\{u<-\delta\}\subset B$$
and $\gamma$ be the unit outward normal on $\partial \Omega_\delta$.
We have
\begin{eqnarray*}
\int_{\Omega_\delta}U^{ij}w_{ij}(u+\delta)\,dx
&=&-\int_{\Omega_\delta}U^{ij}w_{j}u_i\,dx\nonumber\\
&=&-\int_{\partial\Omega_\delta}wU^{ij}u_i\gamma_j\,dS
+\int_{\Omega_\delta}w\det D^2u\, dx\\
&\geq&-\int_{\partial\Omega_\delta}wU^{ij}u_i\gamma_j\,dS\\
&=&-\int_{\partial\Omega_\delta}
wU^{\gamma\gamma}u_\gamma \,dS\\
&=&-\int_{\partial\Omega_\delta}wu_\gamma^nK_s \, dS\\
&\geq & -C \sup_{\partial \Omega_\delta}w \sup_B |Du|^n ,
\end{eqnarray*}
where $dS$ is the volume element of $\partial \Omega_\delta$
and $K_s$ is the Gaussian curvature of $\partial \Omega_\delta$.
Letting $\delta\rightarrow 0$,
by $w=1$ on $\partial B$ and the gradient estimate,
$$\int_{B}f(x, u)u\, dx \geq -C.$$
By a similar argument as in the proof of lower bound,
if $u-\varphi$ is sufficiently close to $1$
at some point $x\in{B\setminus\Omega}$,
$u-\varphi$ is sufficiently close to $1$
nearby in $B\setminus \Omega$. This implies
the integral can be arbitrary large, which is a contradiction.
Hence, $f$ is bounded and the lemma follows.
\end{proof}

Now we prove that the maximizer  of $J(u)$ can be approximated
by smooth solutions to (2.6). This approximation
was proved for the affine Plateau problem in \cite{TW2} by a penalty method.
We will also use this method.

\begin{theo}
Let $\Omega$ and $\varphi$ be as in Theorem 2.6.
Suppose $\partial \Omega$ is Lipschitz continuous.
Then there exist a sequence of smooth solutions to equation (2.6)
converging locally uniformly to the maximizer $u$.
\end{theo}

\begin{proof}
The proof for this approximation in \cite{TW2} is very complicated,
so we use a simplified proof in \cite{TW5}.

Let $B=B_R(0)$ be a large ball such that $\Omega\subset B_R$.
By assumption, $\varphi$ is defined in a neighborhood of
$\Omega$, so we can extend $u$ to $B$ such that
$\varphi$ is convex in $B$, $\varphi\in{C^{0, 1}(\overline B)}$
and $\varphi$ is constant on $\partial B$.
Adding $(|x|-R+\frac{1}{2})^2_+$ to $\varphi$, where
$$(|x|-R+\frac{1}{2})_+=\max\{|x|-R+\frac{1}{2}, 0\},$$
we assume that $\varphi$ is uniformly convex in
$\{x\in{\mathbb R^n}\ | \    R-\frac{1}{2}<|x|<R\}$.
Consider the second boundary value problem (5.2) with
\begin{eqnarray*}
f_{j}(x, u)&=&
\begin{cases}
f & \ \text{in}\  \Omega,\\
H_j'(u-\varphi)& \ \text{in}\  B_R\setminus \Omega,
\end{cases}
\end{eqnarray*}
where $H_j(t)=H(4^j t)$ and $H$ is defined by (5.1).
By Lemma 5.1, there is a solution $u_{j}$ satisfying
\beq
|u_{j}-\varphi|\leq 4^{-j}, \ x\in{B_R\setminus\Omega}.
\eeq
By the convexity, $u_{j}$ sub-converges
to a convex function $\bar u$ in $B_R$ as $j\rightarrow \infty$.
Note that $\bar u=\varphi$ in $B_R\setminus \Omega$.
Hence, $\bar u\in{\overline S[\varphi, \Omega]}$
when restricted in $\Omega$.
We claim that $\bar u$ is the maximizer.


Let $v_j$ be an extension of $u$, given by
$$v_{j}=\sup\{l \ |\  l\in{\Phi_{j}}\},$$
where $\Phi_{j}$ is the set of
linear functions in $B_R$ satisfying
\beqs
l(x)\leq \varphi(x)&& \text{when} \  |x|=R \ \text{or} \ |x|\leq R-\frac{1}{j},
\  \text{and}\   \\
 l(x)\leq u_j(x) &&\text{when} \  R-\frac{1}{j}<|x|<R.
\eeqs
By our assumption, $\varphi$ is uniformly convex
in $B_R\setminus B_{\frac{R}{2}}$. By (5.7),
$|u_{j}-\varphi|\leq 4^{-j}=o(j^{-2})$,
$x\in{B_R\setminus\Omega}$.
So we have
\begin{eqnarray}
v_{j}=u_{ j} &&\ \text{in}\   B_R\setminus B_{R-\frac{1}{2j}},\\
v_{j}=\varphi &&\ \text{in}\ B_{R-\frac{2}{j}}\setminus \Omega, \\
|v_{j}-\varphi|\leq  |u_{j}-\varphi|&& \  \text{in} \
B_{R-\frac{1}{2j}}\setminus B_{R-\frac{2}{j}}:=D_j.
\end{eqnarray}

\vskip 8pt

Now we consider the functional
$$J_{ j}(v)=\int_{B_R} G(\det\partial^2v)\,dx-\int_{\Omega}fv\,dx
-\int_{B_R\setminus\Omega}H_j(v-\varphi)\,dx.$$
Subtracting $G$ by the constant $G(0)$,
we may assume that $G(0)=0$.
Note that $u_{j}$ is the maximizer of $J_{j}$
in $\overline S[u_{j}, B_R]$ and $v_{j}\in{\overline S[u_{j}, B_R]}$.
So we have
$$J_{j}(v_{j})\leq  J_{j}(u_{j}).$$
In the following, we denote by $J_j(v, E)$ the functional $J_j$
over the domain $E$.
By (5.8), we have
\beq
J_{j}(v_{j}, B_{R-\frac{1}{2j}})
\leq  J_{j}(u_{j}, B_{R-\frac{1}{2j}}).
\eeq
By (5.9), (5.10), we obtain
\beq
-\int_{B_{R-\frac{1}{2j}}\setminus \Omega}H_j(u_j-\varphi)\,dx
\leq-\int_{B_{R-\frac{1}{2j}}\setminus \Omega}H_j(v_j-\varphi)\,dx.
\eeq
For any $\epsilon>0$,
by the upper semi-continuity of the functional $A(u)$,
\beqn
\int_{B_{R-\frac{2}{j}}\setminus \Omega}G(\det\partial^2u_j)\,dx
&\leq& \int_{B_{R-\frac{2}{j}}\setminus \Omega}
G(\det\partial^2\varphi)\,dx+\epsilon\nonumber\\
&=&\int_{B_{R-\frac{2}{j}}\setminus \Omega}G(\det\partial^2v_j)\,dx+\epsilon
\eeqn
provided $j$ is large enough. In addition, by (2.9),
\beq
0\leq \int_{D_j} G(\det \partial^2v)\,dx
\leq |D_j|G(|D_j|^{-1}\mu[v](D_j))\rightarrow 0
\eeq
as $j\rightarrow \infty$, where $v=u_{j}$ or $v_{j}$.

Hence, by (5.11)-(5.14) and the upper semi-continuity of the functional $A(u)$,
$$J(u)=J(v_{j})\leq  J(u_{j})+2\epsilon\leq J(\bar u)+3\epsilon.$$
provided $j$ is large enough.
By taking $\epsilon\rightarrow 0$,
this implies $\bar u$ is the maximizer.
By the uniqueness of maximizers in Theorem 2.6,
we obtain $\bar u=u$.
\end{proof}

\begin{rem}
We remark that the above approximation does not holds for
the maximizer of the functional $J_0$.
The reason is that since $\log d$ is not bounded from below,
we do not have the property
$$\left|\int_E \log\det\partial^2 u\, dx\right|\longrightarrow 0, $$
as $|E|\rightarrow 0$.
This is why we introduce the function $G$
and consider the modified functional $J(u)$.
\end{rem}

\vskip 10pt

By Theorem 5.2, for each $k$,
there exists a smooth solutions $u^{(k)}_{j}$ to
\beq
U^{ij}w_{ij}=f,
\eeq
where
\beq
w=G'_k(\det D^2u),
\eeq
which converges locally uniformly to the maximizer $u^{(k)}$ of (2.13).
Then we have
\beq
u^{(k)}_{j}\longrightarrow u_0,  \    j, k\rightarrow \infty.
\eeq
As we explained in Section 3, if $u_0$ is strictly convex,
the interior a priori estimates of $u^{(k)}_j$ will
be independent of $k$ and $j$. Hence, by taking limit, we have the
interior regularity of $u_0$ in $\Omega$.
Moreover,  by the construction of
$G_k$, $u_0$ will be a solution to Abreu's equation (1.1).
Therefore we have

\begin{theo}
Let $u_0$ be as above.
Assume that $f\in{C^{\infty}(\Omega)}$.
Then if $u_0$ is a strictly convex function,
$u_0\in{C^{\infty}(\Omega)}$ and solves (1.7).
\end{theo}

In the last two sections, we will show the strict convexity of $u_0$.

\vskip 10pt

\section{Strict convexity I}

We prove the strict convexity of $u_0$ in dimension $2$.
Let $\mathcal M_0$ be the graph of $u_0$.
If $u_0$ is not strictly convex, $\mathcal M_{0}$
contains a line segment. Let $l(x)$ be a tangent function of $u_0$
at the segment and denote by
$$\mathcal C=\{x\in{\Omega}\ |  \   u_0(x)=l(x)\}$$
the contact set.

\vskip 5pt

We first recall the definition of extreme points.
Let $\Om$ be a bounded convex domain in $\R^n$, $n\ge 2$.
A boundary point $x\in\pom$ is an {\it extreme point} of
$\Om$ if there is a hyperplane $H$ such that $\{x\}=H\cap\pom$, namely
$x$ is the unique point in $H\cap\pom$.

\vskip 5pt

According to the distribution of
extreme points of $\mathcal C$, we consider two cases as follows.

Case (a) $\mathcal C$ has an extreme point $x_0$ which is
an interior point of $\Omega$.

Case (b) All extreme points of $\mathcal C$ lie on $\partial\Omega$.

In this section, we exclude Case (a). 

\begin{prop}
$\mathcal C$ contains no extreme points in the interior of $\Omega$.
\end{prop}

\begin{proof}

We prove this proposition by contradiction arguments as in \cite{TW1}. 
By (5.17), we can choose a sequence of smooth functions $u_k=u_{j_k}^{(k)}$
converging to $u_0$ such that $u_k$ is the solution to (5.15).
Let $\mathcal M_k$ be the graph of $u_k$.
Then $\mathcal M_k$ converges
in Hausdorff distance to $\mathcal  M_0$.
There is no loss of generality in assuming that $l(x)=0$,
$x_0$ is the origin and the segment
$\{(x_1, 0)\  |  \  0\leq x_1\leq 1 \}\subset \mathcal C$.

For any $\epsilon>0$, we consider a linear function
$$l_\epsilon=-\epsilon x_1+\epsilon$$ and a subdomain
$\Omega_\epsilon=\{u< l_\epsilon\}$.
Let $T_\epsilon$ be the coordinates transformation that normalizes
$\Omega_\epsilon$. Define
\beq
u_\epsilon(y)=\frac{1}{\epsilon}u(x), \
u_{k,\epsilon}=\frac{1}{\epsilon}u_k(x),\ \  y\in{\tilde \Omega_\epsilon}
\eeq
where $y=T_\epsilon x$ and
$\tilde \Omega_\epsilon=T_\epsilon(\Omega_\epsilon)$.
After this transformation, we have the following observations:

(i) By Remark 2.2, $u_{k,\epsilon}$ satisfies the equation (2.6) with
$$G=G_{k,\epsilon}(d)=G_k(\epsilon|T_\epsilon|^2d), \  \
\delta=\delta_{k, \epsilon}=\frac{\delta_k}{\epsilon |T_\epsilon|^2}$$
and the right hand term $\epsilon f$.
Note that $|T_\epsilon|\geq C\epsilon^{-1}$, so
$\delta_{k, \epsilon}\leq C\delta_k\rightarrow 0$
for a constant $C$ independent of $\epsilon$.

(ii) Denote by $\mathcal M_\epsilon$, $\mathcal M_{k,\epsilon}$
the graphs of $u_\epsilon$, $u_{k, \epsilon}$, respectively.
Taking $k\rightarrow \infty$, it is clear that
$u_{k, \epsilon}\rightarrow u_\epsilon$
and $\mathcal M_{k,\epsilon}$
converges in Hausdorff distance to
$\mathcal M_\epsilon$. Then taking $\epsilon\rightarrow 0$,
we have that the domains $\tilde \Omega_\epsilon$
sub-converges to a normalized domain $\tilde\Omega$ and
$u_\epsilon$ sub-converges to a convex function $\tilde u$
defined in $\tilde\Omega$. We also have $\mathcal M_\epsilon$
sub-converges in Hausdorff distance
to a convex surface $\tilde{ \mathcal M}_0\in{\mathbb R^3}$.

(iii) The convex surface $\tilde{ \mathcal M}_0$  satisfies
\beq
\tilde{\mathcal M}_0\subset \{y_1\geq 0\}\cap\{y_3\geq 0\}
\eeq
and $\tilde{\mathcal M}_0$ contains two segments
\beq
\{(0, 0, y_3)\ | \   0 \leq y_3\leq 1\}, \
\{(y_1, 0, 0)\  | \  0\leq y_1\leq 1\}.
\eeq

Hence,  by (i), (ii), (iii),
we can suppose that there is a solution
$\tilde u_k$ to
\beq
U^{ij}w_{ij}=\epsilon_k f  \  \ \text{in}\ \tilde \Omega_k,
\eeq
where
\beq
w=\tilde G'_{\tilde \delta_k}(\det D^2u),
\eeq
and $\tilde\delta_k, \epsilon_k\rightarrow 0$,
such that the normalized domain $\tilde \Omega_k$
converges to $\tilde\Omega$,
$\tilde u_k$ converges to $\tilde u$
and the graph of $\tilde u_k$, denoted by
$\tilde {\mathcal M}_k$ converges in
Hausdorff distance to $\tilde{\mathcal M}_0$.
It is clear that in $y$-coordinates, $\tilde{\mathcal M}_0$
is not a graph of a function near the origin, so
we need to rotate the $\mathbb R^3$ coordinates.
Since the equation (2.6) is invariant under
unimodular transformation,
we may suppose
$$\tilde\Omega\subset \{y_1\geq 0\}.$$
Adding a linear function to $\tilde u$, $\tilde u_k$,
we replace (6.2), (6.3) by
\beq
\tilde{\mathcal M}_0\subset \{y_1\geq 0\}\cap\{y_3\geq -y_1\}
\eeq
and $\tilde{\mathcal M}_0$ contains two segments
\beq
\{(0, 0, t)\ | \   0 \leq t\leq 1\}, \
\{(t, 0, -t)\  | \  0\leq t\leq 1\}.
\eeq
Let
$$L=\{(y_1, y_2, y_3)\in{\tilde {\mathcal M}_0} \  |  \
y_1=y_3=0 \}.$$
$L$ must be a single point (Case I) or a segment (Case II).
In Case II, we may also suppose that $0$ is an end point of
the segment which is
$$\{(0, t, 0)\ | \  -1<t<0\}.$$
Later, we will discuss the two cases separately.

\vskip 8pt

Now we make the rotation
$$z_1=-y_3,  \  \  z_2=y_2, \ \ z_3=y_1$$
such that $\tilde{\mathcal M}_0$ can be
represented by a convex $v$ near the origin.
By convexity, $\tilde{\mathcal M}_k$ can also be represented by
$z_{3}=v^{(k)}(z_1, z_2)$ near $p_0$,  respectively.
$v^{(k)}$ is a solution of the equation given in Lemma 4.1 near the origin.
As we know that  $\tilde{\mathcal M}_k$ converges
in Hausdorff distance to $\tilde{\mathcal M}_0$,
in new coordinates,  $v^{(k)}$ converges locally uniformly to $v$.
It is clear that
\begin{eqnarray}
&&v(0)=0, \ \ v\geq 0,\  \text{when}\  -1\leq z_1\leq 0 \ \text{and}\  \nonumber\\
&&v\geq z_1,\ \text{when} \ 0\leq z_1\leq 1\nonumber
\end{eqnarray}
and the two line segments
$$\{(t, 0, 0)\ | \ -1\leq t\leq 0\}, \quad \{(t, 0, t)\ |\ 0\leq t\leq 1\}$$
lie on the graph of $v$.

As in (4.9),   let $\hat v^{(k)}=v^{(k)}-\frac{1}{2}z_1$
and $\hat v=v-\frac{1}{2}z_1$.
In the following computation we omit the hat for simplicity. Then
\beq
v\geq \frac{1}{2}|z_1|\quad \text{and} \quad
v(z_1, 0)=\frac{1}{2}|z_1|.
\eeq
Let $$\tilde{\mathcal{C}}=\{z\ | \ v(z)=0\}.$$
Observe that
$$L=\{(z_1, z_2, 0)\ | \  (z_1, z_2)\in{ \tilde{\mathcal{C}}}\}$$
in $z$-coordinates.

\vskip 10pt

\noindent {\it Case I}.  In this case,  $v$ is strictly convex at $(0, 0)$.
The strict convexity implies that
$Dv$ is bounded on $S_{h, v}(0)$ for small $h>0$.
Hence, by locally uniform convergence,
$Dv^{(k)}$ are uniformly bounded on $S_{\frac{h}{2}, v^{(k)}}(0)$.
By Lemma 4.3, we have the determinant estimate
\beq
\det D^2v^{(k)}\leq C
\eeq
near the origin.

For $\delta\leq \frac{h}{2}$, by (6.8),
$S_{\delta,v}(0)\subset\{-\frac{\delta}{2} \leq y_1\leq \frac{\delta}{2}\}$
 and $(\pm\frac{\delta}{2}, 0)\in{\partial S_{\delta,v}(0)}$.
In the $z_2$ direction, we define
$$\kappa_\delta=\sup\{|z_2|\ |  \  (z_1, z_2)\in{S_{\delta,v}(0)}\}.$$
By comparing the images of $S_{\delta,v}(0)$ under normal mapping of $v$ and
the cone with bottom at $\partial S_{\delta,v}(0)$ and
top at the origin,
$$|N_v(S_{\delta,v}(0))|\geq C\frac{\delta}{\kappa_\delta}.$$
By the lower semi-continuity of normal mapping,
$$N_v(S_{\delta,v}(0))\subseteq
{\lim\inf}_{k\rightarrow \infty}N_{v^k}(S_{\delta,v}(0)),$$
then
$$N_v(S_{\delta,v}(0))
=N_v(S_{\delta,v}(0))
\subseteq {\lim\inf}_{k\rightarrow \infty}
N_{v^{(k)}}(S_{\delta,v}(0)).$$
By (6.9),
\beqn
|N_v(S(\delta))|
&\leq&{\lim\inf}_{k\rightarrow \infty}|N_{v^{(k)}}(S_{\delta,v}(0))|\nonumber\\
&=& {\lim\inf}_{k\rightarrow \infty}
\int_{S_{\delta,v}(0)}\det D^2v^{(k)}\,dz\nonumber\\
&\leq& C |S_{\delta,v}(0)|\nonumber\\
&\leq& C \delta \kappa_\delta.
\eeqn
Hence, $\kappa_\delta\geq C>0$, where $C$ is independent of $\delta$.
Again by the strict convexity,
$\kappa_\delta\rightarrow 0$ as $\delta\rightarrow 0$.
The contradiction follows.

\vskip 10pt

\noindent {\it Case II}.
In this case,
$$\tilde{\mathcal C}= \{(0, z_2)\ | \  -1<z_2<0\}.$$
We define the following linear function:
$$l_\epsilon(z)=\delta_\epsilon z_2+\epsilon$$ and
$\omega_\epsilon=\{z\ | \ v(z)\leq l_\epsilon\}$,
where $\delta_\epsilon$ is chosen
such that
$$v(0,\frac{\epsilon}{\delta_\epsilon})
=l(0,\frac{\epsilon}{\delta_\epsilon})=2\epsilon,\
v(0,-\frac{\epsilon}{\delta_\epsilon})
=l(0,-\frac{\epsilon}{\delta_\epsilon})=0.$$
We can suppose that $\epsilon$ is small enough such that
$\omega_\epsilon$ is contained in a small ball near the origin. Hence,
$Dv^{(k)}$ is uniformly bounded.
By comparing the image of
 $\omega_\epsilon$ under normal mapping of $v$
 and the cone with bottom at $\partial \omega_\epsilon$ and
top at the origin,
\beq
|N_v(\omega_\epsilon)|\geq  C\delta_\epsilon.
\eeq

On the other hand,
$\omega_\epsilon\subset\{-\epsilon \leq z_1\leq \epsilon\}$ since
$v\geq |z_1|$. By the convexity and the assumption above,
$\omega_\epsilon\subset
\{-\frac{\epsilon}{\delta_\epsilon} \leq z_2\leq \frac{\epsilon}{\delta_\epsilon}\}$.
Therefore,
$$|\omega_\epsilon|\leq C\frac{\epsilon^2}{\delta_\epsilon}.$$
Furthermore, subtracting all $v^{(k)}$ by $l_\epsilon$, they
still satisfy the same equation.
By the determinant estimate in Lemma 4.3
and a similar argument as in (6.10),
\beq
|N_v(\omega_\epsilon\cap \{z \ | \ \xi_1\geq 0\})|
\leq C\frac{\epsilon^2}{\delta_\epsilon}.
\eeq
Combining (6.11) and (6.12),
$$\frac{\epsilon^2}{\delta_\epsilon^2}\geq C.$$
However, according to our construction, $\frac{\epsilon}{\delta_\epsilon}$
goes to $0$ as $\epsilon$ goes to $0$. The contradiction follows.
\end{proof}

\begin{rem}
The following property has been used in the above proof.
Assume that $u$ is a 2-dimensional convex function
satisfying
\beq
u(0)=0,\ \ u(x)>0 \ \text{for}\  x\neq 0 \ \text{and} \
u(x_1, 0)\geq C|x_1|.
\eeq
Then
$$\frac{|N_u(S_{h, u}(0))|}{|S_{h, u}(0)|}
\rightarrow \infty \ \text{as}\  h\rightarrow 0.$$
In other words,  if
$$\det D^2u\leq C$$
and $u$ vanishes on boundary, then $u$ is $C^1$ in $\Om$.
This property can be extended to high dimension if
\beq
u(0)=0,\ \ u(x', x_n)\geq C|x_n| \ \text{and} \  u(x', x_n)\geq C|x'|^2,
\eeq
where $x'=(x_1, ..., x_{n-1})$.

It is also known that a generalized solution to
$$\det D^2u \geq C$$
in a domain in $\mathbb R^2$ must be strictly convex.
This result was first proved by Aleksandrov but a simple proof can be found in
\cite{TW3}.
\end{rem}

\vskip 10pt

\section{Strict convexity II}

\vskip 10pt

In this section, we rule out the Case (b) that all extreme points of
$\mathcal C$ lie on the boundary $\partial \Omega$.

First, we need a  stronger approximation. In the case of
the affine Plateau problem, this approximation was obtained
by \cite{TW5}. Here, we extend it to our functional $J(u)$.

\begin{theo}
Let $\varphi$, $\Omega$ be as in Theorem 2.6
and $u$ be the maximizer of the functional $J$
in $\overline S[\varphi, \Omega]$.
Assume that $\partial \Omega$ is lipschitz continuous.
Then there exist a sequence of smooth solutions
$u_m\in{W^{4, p}(\Omega)}$ to
\beq
U^{ij}w_{ij}=f_m=f+\beta_m\chi_{D_m} \ \text{in}\  \Omega
\eeq
such that
\beq
u_m \longrightarrow u \  \  \text{uniformly in}\  \Omega,
\eeq
where $D_m=\{x\in\Omega \  | \
dist(x, \partial \Omega)<2^{-m}\}$,
$\chi$ is the characteristic function, and $\beta_m$ is a constant.
Furthermore, we can choose $\beta_m$ sufficient large
($\beta_m\rightarrow \infty$ as $m\rightarrow \infty$) such that
for any compact subset $K\subset{N_\varphi(\Omega)}$,
\beq
K\subset{N_{u_m}(\Omega)}
\eeq
provided $m$ is sufficient large.
\end{theo}

\begin{proof}
By subtracting the constant $G(0)$,
we assume that $G(0)=0$ and $G\geq 0$.
The proof is divided into four steps.

(i) Let $B=B_R(0)$ be a large ball such that $\Omega\subset B_R$.
By assumption, $\varphi$ is defined in a neighborhood of
$\Omega$, so we can extend $u$ to $B$ such that
$\varphi$ is convex in $B$, $\varphi\in{C^{0, 1}(\overline B)}$
and $\phi$ is constant on $\partial B$.
Consider the second boundary value problem with
\begin{eqnarray*}
f_{m, j} &=&
\begin{cases}
f+\beta_m\chi_{D_m}& \ \text{in}\  \Omega,  \\
H_j'(u-\varphi)& \ \text{in}\  B_R\setminus \Omega,
\end{cases}
\end{eqnarray*}
where $H_j(t)=H(4^j t)$ is given by (5.1).
By Lemma 5.1, there is a solution $u_{m, j}$ satisfying
\beq
|u_{m, j}-\varphi|\leq 4^{-j},
\ x\in{B_R\setminus\Omega}.
\eeq

\vskip 10pt


(ii) By the convexity, $u_{m, j}$ sub-converges to
a convex function $u_{m}$ as $j \rightarrow \infty$ and
$u_{m}=\varphi$ in $B_R\setminus\Omega$.
Note that $u_{m}\in{\overline S[\varphi, \Omega]}$
when restricted in $\Omega$.
By Theorem 5.2, $u_{m}$ is the maximizer of the functional
\beq
J_{m}(v)=\int_{\Omega}G(\det\partial^2v)\,dx-
\int_\Omega (f+\beta_m\chi_{D_m})v\,dx
\eeq
in $\overline S[\varphi, \Omega]$.

\vskip 10pt

(iii) Since $u_m\in{\overline S[\varphi, \Omega]}$,
$u_m$ converges to a convex function $u_\infty$
in $\overline S[\varphi, \Omega]$ as $m\rightarrow \infty$.
We claim that $u_\infty$ is the maximizer $u$.
The proof is as follows.

Define
$$\varphi_{*}=\sup\{l(x) \ | \ l  \  \text{is a tangent plane of
$\varphi$ at some point in $B_R\setminus \overline\Omega$} \}.$$
Then $\varphi_{*}\in{\overline S[\varphi, \Omega]}$
and $v\geq \varphi_{*}$ for any $v\in{\overline S[\varphi, \Omega]}$.
We consider the maximizer $u$. Let
$$\tilde u_m= \sup\{l(x)\ | \ l \ \text{is linear},\
l\leq u \  \text{in}\  \Omega
\  \text{and}\  l\leq \varphi_* \ \text{in}\  D_m \}.$$
Then $\tilde u_m\in{\overline S[\varphi, \Omega]}$
and $\tilde u_m=\varphi_{*}$ in $D_m$.
Since $u$ is convex, it is twice differentiable almost everywhere.
By the definition of $\tilde u_m$, $\tilde u_m=u$ at any point
where $D^2u>0$ when $m$ is sufficiently large.
Therefore, we have
$\det \partial^2\tilde u_m \rightarrow \det\partial^2 u$ a.e..
By the upper semi-continuity of the functional $A(u)$ and Fatou lemma,
$$\lim_{m\rightarrow\infty}
\int_{\Omega}G(\det\partial^2\tilde u_m )\,dx
=\int_\Omega G(\det\partial^2 u)\,dx.$$
It follows that for a sufficiently small $\epsilon_0>0$,
\beq
J(u)\leq J(\tilde u_m)+\epsilon_0
\eeq
provided $m$ is sufficiently large.

On the other hand, we consider the functional $J_{m}$.
By (ii), $u_{m}$ is the maximizer of $J_m$
in $\overline{S}[\varphi, \Omega]$,
so we have
\beq
J_{m}(\tilde u_m)\leq  J_{m}(u_{m}).
\eeq
Note that $u_{m}\geq \varphi_{*}=\tilde u_m$ in $D_m$.
Hence, we obtain
$$\int_{D_m} \beta_m u_{m}\,dx\geq \int_{D_m} \beta_m \tilde u_m\,dx.$$
By the definition of $J_m$, it follows
\beq
J(\tilde u_m)\leq  J(u_{m})+\epsilon_0.
\eeq
for sufficiently large $m$.

Finally, by (7.6), (7.8) and the upper semi-continuity of $A(u)$,
\beqs
J(u)
&\leq& J(\tilde u_m)+\epsilon_0\\
&\leq& J(u_{m})+\epsilon_0\\
&\leq& J(u_{\infty})+2\epsilon_0.
\eeqs
By taking $\epsilon_0\rightarrow 0$,
this implies that $u_\infty$ is the maximizer.
By the uniqueness of maximizers, $u_\infty=u$.

\vskip 10pt

(iv) It remains to prove (7.3). We claim that for any fixed $m$,
\beq
\lim_{\beta_m\rightarrow\infty}u_m(x)\leq \varphi_{*}(x).
\eeq
We prove it by contradiction. Suppose that there is $x_0\in{D_m}$
such that $u_{m}(x_0)\geq \varphi_*(x_0)+\epsilon_0$
for some $\epsilon_0>0$. Since  $u_{m}$ and $\varphi_*$
are uniformly Lipschitz continuous,
$u_{m}(x)\geq \varphi_*(x)+\frac{\epsilon_0}{2}$
in a ball $B_{C\epsilon_0}(x_0)$ for some constant $C$.
Let
$$u_{m*}= \sup\{l(x)\ | \ l \ \text{is linear},\
l\leq u_{m} \  \text{in}\  \Omega
\  \text{and}\  l\leq \varphi_* \ \text{in}\   D_m \}.$$
Then $u_{m*}\in{\overline S[\varphi, \Omega]}$, and satisfies
$$u_{m*}\leq u_{m} \ \text{in}\ \Omega, \ u_{m*}=\varphi_*
\ \text{in}\  B_{C\epsilon_0}(x_0).$$
Hence,
\beqs
J_m(u_m)-J_m(u_{m*})
=J(u_m)-J(u_{m*})
-\beta_m\int_{D_m}u_m-u_{m*}\,dx
\eeqs
becomes  negative when $\beta_m$ is sufficiently large.
This is a contradiction to that $u_m$ is a maximizer of $J_m$.
\end{proof}

\begin{rem}
If $\varphi\in{C^1}$, we can restate (7.3) in the theorem as
\beq
|D(u_m-\varphi)|\rightarrow 0 \  \  \text{uniformly on}\   \partial \Omega.
\eeq
\end{rem}

\vskip 10pt

Now we deal with Case (b). By Theorem 7.1,
there exists a solution $u^{(k)}_{m}$ to
\beq
U^{ij}w_{ij}=f_m,
\eeq
where
\beq
w=G'_k(\det D^2u),
\eeq
such that
$$u^{(k)}_m\longrightarrow u^{(k)}, \  m\rightarrow \infty.$$
and for any compact set $K\subset D\varphi(\Omega)$,
\beq
K\subset Du_m^{(k)}(\Omega)
\eeq
for large $m$.
Hence, we can choose a sequence $m_k\rightarrow \infty$
such that
\beq
u_k:= u^{(k)}_{m_k}\longrightarrow u_0.
\eeq

\begin{lem}
Assume that $\Omega$ and $\varphi$ are smooth.
Then $\mathcal M_{0}$ contains no line segments with both
endpoints on $\partial \mathcal M_{0}$.
\end{lem}

\begin{proof}
Suppose that $L$ is a line segment in $\mathcal M_0$ with both
end points on $\partial \mathcal M_0$.
By subtracting a linear function, we suppose that $u_0\geq 0$
and $l$ lies in $\{x_3=0\}$.
By a translation and a dilation of the coordinates, we may further
assume that
\beq
L=\{(0, x_2, 0)\ | \   -1\leq x_2\leq 1\}
\eeq
with $(0, \pm 1)\in \partial \Omega$. Note that by
Remark 2.2, these transformations do not change the
essential properties of equation (2.6).

Since $\varphi$ is a uniformly convex function in a neighborhood of
$\Omega$ and $\varphi=u_0$ at $(0, \pm 1)$, $L$ must be
transversal to $\partial \Omega$ at $(0, \pm 1)$.
Hence, by $u_0=\varphi$ on $\partial \Omega$
and the smoothness of $\varphi$ and $\partial \Omega$, we have
$$u_0(x)=\varphi(x)\leq \frac{C}{2}x_1^2,  \  x\in{ \partial \Omega}.$$
By the convexity of $u_0$,
\beq
u_0(x)\leq \frac{C}{2}x_1^2,  \  x\in{\Omega}.
\eeq

Now we consider the Legendre function $u_0^*$
of $u_0$ in $\Omega^*=D\varphi(\Omega)$, given by
$$u_0^*(y)=\sup\{x\cdot y-u_0(x), x\in\Omega\}, \ y\in{\Omega^*}.$$
Note that
$(0, \pm 1)\in \partial \Omega$.
By the uniformly convexity of $\varphi$,
$0\notin D\varphi(\partial \Omega)$.
Hence, $0\in {\Omega^*}$.
By (7.15), (7.16) and the smoothness of $\varphi$, we have
\beqn
u^*(0, y_2)&\geq& |y_2|,\\
u^*(y)&\geq & \frac{1}{2C}y_1^2.
\eeqn

On the other hand, by the approximation (7.13), (7.14),
the Legendre function of $u_k$, denoted by $u_k^*$, is smooth in
$$\Omega^*_{\epsilon_k}
=\{y \in{\Omega^*\  |  \   dist(y, \partial \Omega^*)>\epsilon_k}\}.$$
 with
$\epsilon_k\rightarrow 0$ as $k\rightarrow \infty$ 
and satisfies the equation
\beq
{u^*}^{ij}{w^*}_{ij}=-f_{m_k}(Du^*)
\eeq
in $\Omega^*_{\epsilon_k}$, where
\beq
{w^*}=G_k({d^*}^{-1})-{d^*}^{-1}G_k'({d^*}^{-1}).
\eeq

By (7.17), (7.18), $u_0^*$ is strictly convex at $0$. Then
$\{y\ | \ u^*_0<h\}\subset \Omega^*_{\epsilon_k}$ providing
$m$ is sufficiently large. 
Note that $u_k^*$ converges to $u_0^*$.
By Lemma 3.2, we have the estimate
$$\det D^2u_k^*\leq C$$
near the origin in $\Omega^*$. Note also that in Lemma 3.2,
$C$ depends on $\inf f$ but not on $\sup f$.
In other words, the large constant $\beta_{m_k}$ in (7.1)
does not affect the bound $C$.
Therefore sending $k\rightarrow \infty$, we obtain
$$\det D^2u^*_0\leq C$$
in the sense that the Monge-Amp\`ere measure of $u^*_0$ is
an $L^\infty$ function.
This is a contradiction with (7.17), (7.18) according to Remark 6.2.
\end{proof}

In conclusion, we have proved that $u_0$ is strictly convex in $\Omega$
in dimension $2$. Theorem 1.1 follows from Theorem 5.4.

\vskip 10pt

\section{Appendix: Second boundary value problem}

\vskip 10pt

In order to construct approximation solutions to the maximizer of $J(u)$,
we employ the second boundary value problem for equation (2.6).
This section is just a modification of the second boundary problem in \cite{TW2}.
We include it here for completeness.
Throughout this section, we will denote by $d$
the determinant $\det D^2u$ for simplicity.

We study the existence of smooth solutions to the following problem.
\begin{eqnarray}
&&U^{ij} w_{ij}=f(x, u), \ \text{in}\ \Omega, \\
&&w=G'(d), \ \text{in}\ \Omega,\\
&&w=\psi, \ \text{on}\  \p \Omega,\\
&&u=\varphi, \ \text{on}\  \p \Omega,
\end{eqnarray}
where $\Omega$ is a smooth, uniformly convex domain in $\Bbb{R}^n$,
$\varphi$, $\psi$ are smooth functions on $\p \Omega$ with
$$0<C_0^{-1}\leq \psi \leq C_0.$$
$f\in L^\infty(\Omega\times R)$ is nondecreasing in $u$ and there is $t_0\leq 0$
such that
$$f(x, t)\leq 0,\  t\leq t_0.$$
We note that this condition is not needed if $u$
is bounded from below.

By Inverse Function Theorem,
$w=G'(d)$ has an inverse function $d=g(w)$.
$g$ is an decreasing function which goes to $0$
as $w\rightarrow \infty$ and goes to $\infty$ as $w\rightarrow 0$.
To solve the problem (8.1)-(8.4),
we first consider the approximating problem
\begin{eqnarray}
&&U^{ij} w_{ij}=f, \ \text{in}\ \Omega, \\
&&\det D^2u=\eta_k g(w)+(1-\eta_k), \ \text{in}\ \Omega,
\end{eqnarray}
where $\varphi$ and $\psi$ satisfy (8.3), (8.4) and
$\eta_k\in{C_0^\infty(\Omega)}$ is the
cut-off function satisfying $\eta_k=1$ in
$\Omega_k=\{x\in \Omega \ |\ dist(x, \p \Omega)> \frac{1}{k}\}$.

\begin{lem}
Suppose that $f\in{L^{\infty}}$ satisfies the condition above.
If $(u, w)$ is the $C^2$ solution of (8.5), (8.6),
there is a constant depending only on $diam(\Omega)$,
$f$, $\varphi$, $\psi$ and independent of $k$, such that
\begin{eqnarray}
&&C^{-1}\leq w\leq C, \ \text{in}\ \Omega, \\
&&|w(x)-w(x_0)|\leq C|x-x_0|, \ \text{for any}\ x\in{\Omega}, x_0\in{\p \Omega}.
\end{eqnarray}
\end{lem}

\begin{proof}

The proof of the upper bound for $w$ is totally the same as that for
affine maximal surface equation in \cite{TW2}
by considering the auxiliary function
$$z=\log w+A|x|^2,$$
where $A>0$ is a constant to be determined later.
Suppose that $z$ attains its minimum at the point $x_0$.
If $x_0$ is a boundary point,
then $z(x_0)\geq C$, and hence $w\geq C$. If $x_0$ lies in the interior of $\Omega$,
we have, at $x_0$,
\begin{eqnarray}
&&0=z_i=\frac{w_i}{w}+ 2Ax_i, \nonumber\\
&&0\geq z_{ij}=\frac{w_{ij}}{w}-\frac{w_iw_j}{w^2}+2A\delta_{ij}. \nonumber\end{eqnarray}
Then
\beqs
0&\geq &u^{ij}z_{ij}\\
&=&\frac{f}{dw}-4A^2u^{ij}x_ix_j+2Au^{ii}\\
&\geq & \frac{f}{dw}+Ad^{-\frac{1}{n}}.
\eeqs
Note here we choose $A$ small. Therefore,
$$d^{\frac{n-1}{n}}w\leq C.$$
combining with the definition of $d$, $w$
and using the condition $F'(0)=\infty$ in (c), we obtain $w\leq C$.

By $w\leq C$, we have $\det D^2u \geq C$.
Suppose that $v$ is a smooth, uniformly convex
function such that $D^2v\geq K>0$ and $v=\psi$ on $\p \Omega$.
Then, if $K$ is large,
$$U^{ij}v_{ij}\geq KU^{ii}\geq K[\det D^2v]^{\frac{n-1}{n}}\geq CK\geq f,$$
which implies $U^{ij}(v-w)_{ij}\geq 0$.
By maximum principle, $v-w\leq 0$. We thus obtain
\beq
w(x)-w(x_0)\geq -C|x-x_0|,
\ \text{for any}\ x\in{\Omega},\   x_0\in{\p \Omega}.
\eeq

\vskip 8pt

To prove the lower bound of $w$, let
$$z=\log w+w-\alpha h(u),$$
where $\alpha>0$ is a constant to be determined later and
$h$ is a convex, monotone increasing function such that,
$$h(t)=t, \ \text{when}\ t\geq -t_0 \ \text{and}\  h\geq -t_0-1,
\ \text{when}\ t\leq -t_0.$$
Assume that $z$ attains its minimum at $x_0$.
If $x_0$ is near $\p \Omega$, by (8.9),
$z(x_0)\geq  -C.$
Otherwise, $x_0$ is away from the boundary.
Hence, we have, at $x_0$,
\beqs
&&0=z_i=\frac{w_i}{w}+ w_i-\alpha h'(u)u_i, \\
&&0\leq z_{ij}=\frac{w_{ij}}{w}-\frac{w_iw_j}{w^2}+
w_{ij}-\alpha h''(u)u_iu_j-\alpha h'(u)u_{ij}.
\eeqs
By maximum principle,
\beqs
0\leq u^{ij}z_{ij}
&=&\frac{f}{dw}-\frac{u^{ij}w_iw_j}{w^2}+\frac{f}{d}
-\alpha h''(u)u^{ij}u_iu_j-\alpha h'(u)n\\
&\leq &\frac{f}{dw}+\frac{f}{d}-\alpha h'(u)n.
\eeqs
If $u(x_0)\leq t_0$, $f\leq 0$, which immediately induces a contradiction.
Hence, $u(x_0)\geq t_0$, and $h'(u(x_0))\geq h'(t_0)$.
Then choosing $\alpha$ large enough,
we obtain $d\leq C$ at $x_0$ by the assumption (a).
Using the relation between $w$ and $d$,
we have $w(x_0)\geq C$.
By definition,
$$z=\log w+w-\alpha h(u)\geq z(x_0)\geq -C.$$
This implies $w\geq C$.

Similarly, with the upper bound of the determinant,
we can construct a barrier function $v$ from above for $w$ and prove
$$w(x)-w(x_0)\leq C|x-x_0|.$$

In conclusion, the lemma has been proved.
\end{proof}

\begin{prop}
There is a solution
$u\in{C^{2, \alpha}(\overline\Omega)\cap W^{4, p}(\Omega)}$
to the approximation problem (8.5), (8.6). If furthermore
$f\in{C^\alpha(\overline\Omega)}$, then
$u\in{C^{4, \alpha}(\overline\Omega)}$.
\end{prop}

\begin{proof}
By (8.7), using Caffarelli-Gutierrez's H\"older continuity
for linearized Monge-Amp\`ere equation \cite {CG}
we have the interior $C^\alpha$ estimate for $\det D^2u$,
for some $\alpha\in{(0, 1)}$.
Then by Caffarelli's $W^{2, p}$ and $C^{2,\alpha}$ estimates
for Monge-Amp\`ere equation \cite{Caf, JW}, we have
interior $W^{2, p}$ estimate for $u$ for some $p>1$ and
$C^{2, \alpha}$ estimate when $f\in{C^\alpha(\overline\Omega)}$.
Then the interior $W^{4, p}$ and $C^{4,\alpha}$ estimates
follow from the standard elliptic regularity theory.
Note that $\det D^2u$ is constant near the boundary of $\Omega$,
we also have the boundary $W^{4, p}$ and $C^{4,\alpha}$ estimates
by \cite{CNS, GT, K}.
In conclusion, we have
\beq
\|u\|_{W^{4, p}(\Omega)}\leq C,
\eeq
where $C$ depends on $n$, $p$, $\varphi$, $\psi$ and $f$.
and
\beq
\|u\|_{C^{4, \alpha}(\overline\Omega)}\leq C
\eeq
when $f\in{C^\alpha(\overline\Omega)}$,
where $C$ depends on $n$, $\alpha$, $\varphi$, $\psi$ and $f$.

Now we use the degree theory to prove the existence of
solutions to the approximating problem (8.5), (8.6).

For any positive $w\in{C^{0, 1}(\overline \Omega)}$,
let $u = u_w$ be the solution of (8.6) with $u = \varphi$ on $\partial \Omega$.
Next, let $w_t$, $t\in{[0, 1]}$, be the solution of
\beq
U^{ij}w_{ij}=t f \   \text{in}\  \Omega, \
w_t=t\psi +(1-t) \  \text{on}\  \partial\Omega.
\eeq
Therefore, we have a compact mapping
$$T_t: w\in{C^{0, 1}(\overline\Omega)}\longrightarrow
w_t\in{C^{0, 1}(\overline\Omega)}.$$
By estimate (8.10), the degree $deg(T_t, B_R, 0)$ is well defined,
where $B_R$ is the set of all functions satisfying
$\|w\|_{C^{0, 1}(\overline\Omega)}\leq R$.
When $t=0$, $T_0$ has a unique fixed point $w=1$ by (8.12).
Hence, $deg(T_0, B_R, 0)=1$. By degree theory, we have
$deg(T_1, B_R, 0)=1$. Namely, there is a unique solution when $t=1$.
The proposition follows.
\end{proof}

Finally, taking $k\rightarrow \infty$, we obtain

\begin{theo}
The second boundary problem (8.1)-(8.4) admits a solution
$u\in{W^{2,p}_{loc}}\cap C^{0,1}(\overline \Omega)$$(p>1)$
with $\det D^2u \in{C^0(\overline \Omega)}$.
Moreover, if $f\in{C^\alpha(\overline \Omega\times \Bbb{R})}$
$(0<\alpha<1)$,
then $u\in{C^{4,\alpha}(\Omega)\cap C^{0,1}(\overline \Omega)}$.
\end{theo}

\begin{rem}
The second boundary problem we consider here
is for the equation (2.6).  By checking the proof, it is easy to see that Theorem 8.3
also holds for Abreu's equation.
\end{rem}


\end{document}